\documentclass[11pt]{amsart}
\usepackage{amsmath,amssymb,amsthm,graphicx}
\usepackage{hyperref}
\usepackage[usenames]{color}
\usepackage[shortlabels]{enumitem}

\newtheorem{thm}{Theorem}
\newtheorem{lem}[thm]{Lemma}

\theoremstyle{definition}
\newtheorem{defn}[thm]{Definition}

\newtheorem{rmk}[thm]{Remark}
\newtheorem{exmp}[thm]{Example}

\newcommand{\CPb}{\overline{\mathbb{CP}}{}^{2}}
\newcommand{\CP}{{\mathbb{CP}}{}^{2}}

\newcommand{\Z}{\mathbb{Z}}
\newcommand{\scparallel}{{\scriptscriptstyle \parallel}}

\title[Genus two Lefschetz fibrations with $b^{+}_{2}=1$ and ${c_1}^{2}= 1 , 2$]
{Genus two Lefschetz fibrations with $b^{+}_{2}=1$ and ${c_1}^{2}= 1, 2$} 

\begin{document}

\author{Anar Akhmedov}
\address{School of Mathematics, 
University of Minnesota, 
Minneapolis, MN, 55455, USA}
\email{akhmedov@math.umn.edu}

\author{Naoyuki Monden}
\address{Department of Engineering Science, 
Osaka Electro-Communication University, 
Hatsu-cho 18-8, Neyagawa, 572-8530, Japan}
\email{monden@isc.osakac.ac.jp}

\date{June 15, 2015}



\begin{abstract} In this article we construct a family of genus two Lefschetz fibrations $f_{n}: X_{\theta_n} \rightarrow \mathbb{S}^{2}$ with $e(X_{\theta_n})=11$, $b^{+}_{2}(X_{\theta_n})=1$, and $c_1^{2}(X_{\theta_n})=1$ by applying a single lantern substitution to the twisted fiber sums of Matsumoto's genus two Lefschetz fibration over $\mathbb{S}^2$. Moreover, we compute the fundamental group of $X_{\theta_n}$ and show that it is isomorphic to the trivial group if $n = -3$ or $-1$, $\mathbb{Z}$ if $n =-2$, and $\mathbb{Z}_{|n+2|}$ for all integers $n\neq -3, -2, -1$. Also, we prove that our fibrations admit $-2$ section, show that their total space are symplectically minimal, and have the symplectic Kodaira dimension $\kappa = 2$. In addition, using the techniques developed in \cite{A, AP1, ABP, AP2, AZ, AO}, we also construct the genus two Lefschetz fibrations over $\mathbb{S}^2$ with $c_1^{2} = 1, 2$ and $\chi = 1$ via the fiber sums of Matsumoto's and Xiao's genus two Lefschetz fibrations, and present some applications in constructing exotic smooth structures on small $4$-manifolds with $b^{+}_{2} = 1$ and $b^{+}_{2} = 3$.
\end{abstract}

\maketitle

\section{Introduction}

The Lefschetz fibrations are fundamental objects to study in $4$-dimensional topology and they serve as a powerful tool for understanding the geometry and topology of symplectic $4$-manifolds. It was shown by S. Donaldson that every closed symplectic $4$-manifold admits a Lefschetz pencil, which can be blown up at its base points to yield a Lefschetz fibration \cite{D1}, and conversely, R. Gompf has shown that the total space $X$ of a genus $g$ Lefschetz fibration admits a symplectic structure, provided that the homology class of the regular fiber is nonzero in $H_{2}(X, \mathbb{R})$ \cite{GS}.  

Let us recall from \cite{GS} that, given a Lefschetz fibration $f: X \rightarrow \mathbb{S}^2$, one can associate to it a word $W_{f} = 1$ in the mapping class group of the regular fiber composed only from right-handed Dehn twists. Conversely, given such a right handed Dehn twist factorization $t_{\gamma_1} t_{\gamma_2} \cdots t_{\gamma_n} = 1$ in the mapping class group $M_{g}$ of the closed orientable surface of genus $g$, one can explicitly construct genus $g$ Lefschetz fibration over $\mathbb{S}^2$ with the given monodromy.

In recent literature (\cite{EG, EMVHM, AP, AO, AZ, AMon}), several authors have applied various relations in the mapping class group or Luttinger surgery to construct new Lefschetz fibrations with $1 \leq b^{+}_{2} \leq 3$ and $c_1^{2} \geq 0$. For example, the exotic genus two Lefschetz fibrations with $b^{+}_{2} = 3$ and $0 \leq c_1^{2} \leq 4$ obtained via lantern relations can be found in \cite{AP}. The Lefschetz fibrations over $\mathbb{S}^2$ with $b^{+}_{2} = 1, 3$ and  $c_1^{2} = 0$, and with various abelian fundamental groups, obtained via Luttinger surgery on the fiber sum of $E(1)$ with $\mathbb{T}^2 \times \mathbb{T}^2$, can be found in \cite{AZ, AO}. However, the exotic symplectic $4$-manifolds obtained in \cite{EMVHM, AP, AO, AMon} do not have very small topology as one may desire. Some Lefschetz fibrations constructed in \cite{EG, AMon} have  $b^{+}_{2} = 1$, but with ${c_1}^{2} \leq 0$, and it was left open in \cite{EG, AMon} whether total spaces of these Lefschetz fibrations are exotic symplectic $4$-manifolds. Thus, it is an interesting problem to construct the Lefschetz fibrations with $b^{+}_{2} = 1$ and $c_1^{2} > 0$ (with total space being non-standard or exotic symplectic $4$-manifolds), which was our main motivation for writing this article. Our other motivations comes from the study of genus two Lefschetz fibration with small number of singular fibers, and the relationship between the minimality and the fiber-sum decomposability of Lefschetz fibrations (see the Remark~\ref{decomp}). 

In this artcle, we will consider the lantern relation, a particularly well understood relation in the mapping class group. It was shown in \cite{EG} that the lantern substitution corresponds to the symplectic rational blowdown surgery along $-4$ sphere. We will first construct a family of genus two Lefschetz fibrations $f_{n}: X_{\theta_n} \rightarrow \mathbb{S}^{2}$ with $e(X_{\theta_n})=11$, $b^{+}_{2}(X_{\theta_n})=1$, $c_1^{2}(X_{\theta_n})=1$, and $\pi_1(X_{\theta_n}) \cong \mathbb{Z}_{|n+2|}$ (for any integer $n$) from the twice symplectic fiber sum of Matsumoto's well known genus two Lefschetz fibration with the monodromy $\iota^{2} = (B_0B_1B_2B_3)^2 = 1$, obtained by factorizing the class of vertical involution $\iota$ of the genus two surface with two fixed points in the mapping class group $M_2$ of the closed orientable surface of genus $2$ (see \cite{Ma}). More precisely, our genus two Lefschetz fibrations $f_{n}: X_{\theta_n} \rightarrow \mathbb{S}^{2}$ will be obtained by applying the conjugations and a single lantern substitution to the word $(\iota^{2})^{2} = (B_0B_1B_2B_3)^4 = 1$ in $M_2$. Furthemore, we will show that the total spaces of our Lefschetz fibraions are minimal symplectic $4$-manifolds of general type (i.e. with the symplectic Kodaira dimension $\kappa = 2$), and prove that our fibration admits $-2$ section. Note that in the cases, when the fundamental group of $X_{\theta_n}$ is trivial (i.e. when $n = -3, -1$), the total space of our Lefschetz fibration is an exotic copy of $\CP\#8\CPb$. For $n \notin \{-7, -6, -5, -3, -1, 0, 1, 2, 3\}$, the total space $X_{\theta_n}$ is non-K\"{a}hler symplectic $4$-manifold. This follows from the restriction on the fundamental group of minimal K\"{a}hler surfaces with the invariants $c_1^{2} = \chi = 1$.

In the last section of our article, we also construct a family of Lefschetz fibrations over $\mathbb{S}^2$ via the symplectic fiber sums of Matsumoto's and Xiao's \cite{Xiao} genus two Lefschetz fibrations with $c_1^{2} = 1, 2$, $\chi = 1$ and with various abelian fundamental groups, provide constructions of small exotic Lefschetz fibrations with $b_{2}^{+}=1, 3$, and give further applications for constructing exotic smooth structures on small $4$-manifolds. In the cases, when the fundamental group of our fibrations are trivial, the total space of our Lefschetz fibrations are exotic copies of $\CP\#7\CPb$, $\CP\#8\CPb$, $3\CP\#12\CPb$, $3\CP\#13\CPb$, $3\CP\#14\CPb$, and $3\CP\#15\CPb$. Our construction method and building blocks are similar to the ones \cite{A, AP1, ABP, AP2, AZ, AO}, but our constructions requires a good understanding of Xiao's beautiful genus two Lefschetz fibration in \cite{Xiao}, which we present in Example~\ref{XF}. In the sequel \cite{AM}, we construct higher genus Lefschetz fibrations over $\mathbb{S}^{2}$, using the lantern, daisy and other relations in the mapping class group.

The organization of our paper is as follows. In the next sections, we recall some preliminary definitions and results to be used in the sequel. In Section 3,  we state and prove some of our main results, which are Theorem \ref{thm}, Theorem \ref{thm1}, Theorem \ref{thm2}, and make some remarks. In Section~\ref{section}, we prove the existence of $-2$ sections of our Lefschetz fibrations, which is useful for the fundamental group computations. In the final Section~\ref{twisted}, we construct a family of Lefschetz genus two fibrations over $\mathbb{S}^2$, with the invariants $c_1^{2} = 1, 2$, $\chi = 1$, and the fundamental groups listed, via the symplectic fiber sums of Matsumoto's and Xiao's genus two Lefschetz fibrations over $\mathbb{S}^2$. We also the constructions of small genus two Lefschetz fibrations over $\mathbb{S}^2$ with $b_2^+ = 3$ and $b_2^- = 12, 13, 14, 15$ with a trivial fundamental groups, and provide new constructions of exotic smooth structures on $\CP\#4\CPb$ and $3\CP\#6\CPb$ using ideas of \cite{AP2}, where such examples were first constructed by the first author and D. Park.

\section{Preliminaries}

In this section,  we collect some preliminary definitions and results from \cite{AMon} concerning the mapping class groups, lantern relations and lantern substitutions, Lefschetz fibrations, some details on Matsumoto's genus two Lefschetz fibration on $\mathbb{T}^{2} \times \mathbb{S}^{2}\,\#4\CPb$ over $\mathbb{S}^{2}$ from \cite{Mar, GS, OS} and on Xiao's genus two fibration over $\mathbb{S}^{2}$ \cite{Xiao}, the symplectic Kodaira dimension and the symplectic minimality following \cite{Li2006, li, U, Dor}. Although, we are only interested in genus two Lefschetz fibrations in this article, it will be convenient to state the following definitions in a general setting.

 \subsection{Mapping Class Groups} Let $\Sigma_{g}^k$ be an oriented $2$-dimensional, compact, and connected surface of genus $g$ with $k$ boundary components. 
Let $Diff^{+}\left( \Sigma_{g}^k\right)$ be the group of all orientation-preserving self-diffeomorphisms of $\Sigma_{g}^k$ which are the identity on the boundary and $ Diff_{0}^{+}\left(\Sigma_{g}\right)$ be the subgroup of $Diff^{+}\left(\Sigma_{g}\right)$ consisting of all orientation-preserving self-diffeomorphisms that are isotopic to the identity. 
We assume that the isotopies fix the points on the boundary. \emph{The mapping class group} $M_{g}^k$ of $\Sigma_{g}^k$ is defined to be the group of isotopy classes of orientation-preserving diffeomorphisms of $\Sigma_{g}^n$, i.e.,
\[
M_{g}^k=Diff^{+}\left( \Sigma_{g}^k\right) /Diff_{0}^{+}\left(
\Sigma_{g}^k\right) .
\]
For simplicity, let us set $\Sigma_g = \Sigma_g^0$ and $M_g = M_g^0$.

\begin{defn} Let $\gamma$ be a simple closed curve on $\Sigma_{g}^k$. A \emph{right handed} (or positive) \emph{Dehn twist} about $\gamma$ is a diffeomorphism of $t_{\gamma}: \Sigma_{g}^k\rightarrow \Sigma_{g}^k$ obtained by cutting the surface $\Sigma_{g}^k$ along $\gamma$ and gluing the ends back after rotating one of the ends $2\pi$ to the right. 
\end{defn}

\noindent By a classical theorem of Dehn and Lickorish \cite{FM}, the mapping class group $M_g^k$ is generated by a finite number of Dehn twists.  The following lemma is easy to prove. For the proof, we refer the reader to \cite{I}.

\begin{lem} \label{con&braid.lem} 
\begin{enumerate}[(i)] \item If $\phi: \Sigma_{g}^k\rightarrow \Sigma_{g}^k$ is an orientation-preserving diffeomorphism, then $\phi \circ t_\alpha \circ \phi^{-1} = t_{\phi(\alpha)}$.

\item Let $\alpha$ and $\beta$ be two simple closed curves on $\Sigma_{g}^k$. If $\alpha$ and $\beta$ are disjoint, then their corresponding Dehn twists satisfy the commutativity relation: $t_{\alpha}t_{\beta}=t_{\beta}t_{\alpha}.$ If $\alpha$ and $\beta$ transversely intersect at a single point, then their corresponding Dehn twists satisfy the braid relation: $t_{\alpha}t_{\beta}t_{\alpha}=t_{\beta}t_{\alpha}t_{\beta}.$

\end{enumerate}
\end{lem}

\subsection{Lantern relation and lantern substitution}
Let us recall the definition of the lantern relation (see \cite{FM, Mar, EG}).

\begin{defn}\label{daisy}\rm
Let $\Sigma_0^{4}$ denote a sphere with $4$ boundary components. Let $\delta_0, \delta_1, \delta_2, \delta_{3}$ be the $4$ boundary curves of $\Sigma_0^{4}$ and let $x_1, x_2, x_{3}$ be the interior curves as shown in Figure~\ref{daisy}. Then, we have the \textit{lantern relation}:
\begin{align*}
t_{\delta_0}t_{\delta_1}t_{\delta_2}t_{\delta_{3}}=t_{x_1}t_{x_2} t_{x_{3}}.
\end{align*}
We call the following relator the \textit{lantern relator}: 
\begin{align*}
t_{\delta_{3}}^{-1}t_{\delta_2}^{-1}t_{\delta_1}^{-1}t_{\delta_0}^{-1}t_{x_1}t_{x_2} t_{x_{3}} \ (=1).
\end{align*}
\end{defn}

\begin{figure}[ht]
\begin{center}
\includegraphics[scale=.83]{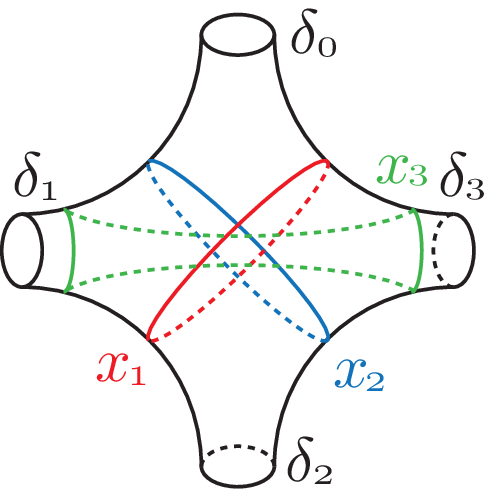}
\caption{Lantern relation}
\label{fig:hyper}
\end{center}
\end{figure}

Next, we introduce a lantern substitution. 

\begin{defn}\label{daisy substitution}\rm
Let $d_1,\ldots,d_m$ and $e_1,\ldots, e_n$ be simple closed curves on $M_g^k$, and let $R$ be a product $R=t_{d_1}t_{d_2}\cdots t_{d_l}t_{e_m}^{-1}\cdots t_{e_2}^{-1}t_{e_1}^{-1}$. Suppose that $R=1$ in $M_g^k$. Let $\tau$ be a word in $M_g^k$ including $t_{d_1}t_{d_2}\cdots t_{d_l}$ as a subword: 
\begin{align*}
\tau=U\cdot t_{d_1}t_{d_2}\cdots t_{d_l} \cdot V, 
\end{align*}
where $U$ and $V$ are words. Thus, we obtain a new word in $M_g^k$, denoted by $\tau^\prime$, as follows:
\begin{align*}
\tau^\prime:&=U\cdot t_{e_1}t_{e_2}\cdots t_{e_m} \cdot V.
\end{align*}
Then, we say that $\tau^{\prime}$ is obtained by applying a $R$-\textit{substitution} to $\tau$. In particular, if $R$ is a lantern relator, then we say that $\tau^{\prime}$ is obtained by applying a \textit{lantern substitution} to $\tau$. 
\end{defn}

\subsection{Lefschetz fibrations}
\begin{defn}\label{LF}\rm
Let $X$ be a closed, oriented smooth $4$-manifold. A smooth map $f : X \rightarrow \mathbb{S}^2$ is a genus-$g$ \textit{Lefschetz fibration} if it satisfies the following condition: \\
(i) $f$ has finitely many critical values $b_1,\ldots,b_m \in S^2$, and $f$ is a smooth $\Sigma_g$-bundle over $\mathbb{S}^2-\{b_1,\ldots,b_m\}$, \\
(ii) for each $i$ $(i=1,\ldots,m)$, there exists a unique critical point $p_i$ in the \textit{singular fiber} $f^{-1}(b_i)$ such that about each $p_i$ and $b_i$ there are local complex coordinate charts agreeing with the orientations of $X$ and $\mathbb{S}^2$ on which $f$ is of the form $f(z_{1},z_{2})=z_{1}^{2}+z_{2}^{2}$, \\
(i\hspace{-.1em}i\hspace{-.1em}i) $f$ is relatively minimal (i.e. no fiber contains a $(-1)$-sphere.)
\end{defn}

Each singular fiber of the Lefschetz fibration is an immersed surface with a single transverse self-intersection, and obtained by collapsing a simple closed curve (the \textit{vanishing cycle}) in the regular fiber. If the curve is nonseparating, then the singular fiber is called \textit{nonseparating}, otherwise it is called \textit{separating}. Moreover, a singular fiber can be described by its monodromy, i.e., by a right handed Dehn twist along the corresponding vanishing cycle. 

A Lefschetz fibration $f : X \rightarrow \mathbb{S}^2$ is \textit{holomorphic} if there are complex structures on both $X$ and $\mathbb{S}^2$ with holomorphic projection $f$. 

For a genus-$g$ Lefschetz fibration over $\mathbb{S}^2$, the product of right handed Dehn twists $t_{\gamma_i}$ along the vanishing cycles $\gamma_i$, for $i = 1,\ldots, m$, determines the global monodromy of the Lefschetz fibration, the relation $t_{\gamma_1} t_{\gamma_2} \cdots t_{\gamma_m}=1$ in $M_g$. This relation is called the \textit{positive relator}. Conversely, for any such a positive relator in $M_g$ one can construct a genus-$g$ Lefschetz fibration over $\mathbb{S}^2$ with the vanishing cycles $\gamma_1,\ldots, \gamma_m$: start with $\Sigma_{g} \times \mathbb{D}^{2}$, attach 2-handles $H_{i}$ along the vanishing cycles $\gamma_{i}$ in a $\Sigma_{g}$-fibers in $\Sigma_{g} \times \mathbb{D}^{2}$ with relative framing $-1$, and attach another copy of $\Sigma_{g} \times \mathbb{D}^{2}$ to $\Sigma_{g} \times \mathbb{D}^{2} \bigcup_{i=1}^{m} H_{i}$ via $\Sigma_{g}$-fiber preserving map of the boundary. We will denote a Lefschetz fibration associated to a positive relator $\varrho \ \in M_g$ by $f_\varrho$.


For a Lefschetz fibration $f:X\rightarrow \mathbb{S}^2$, a map $\sigma:\mathbb{S}^2\rightarrow X$ is called a \textit{section} of $f$ if $f\circ \sigma={\rm id}_{\mathbb{S}^2}$. 
The self-intersection of $\sigma$ is defined to be the self-intersection number of the homology class $[\sigma(\mathbb{S}^2)]$ in $H_2(X;\Z)$. 
Let $\delta_1,\delta_2,\ldots,\delta_k$ be $k$ boundary curves of $\Sigma_g^k$. 
If there exists a lift of a positive relator $\varrho = t_{\gamma_1} t_{\gamma_2} \cdots t_{\gamma_m} = 1$ in $M_g$ to $M_g^k$ as 
\begin{align*}
t_{\tilde{\gamma}_1} t_{\tilde{\gamma}_2} \cdots t_{\tilde{\gamma}_m} = t_{\delta_1} t_{\delta_2} \cdots t_{\delta_k}, 
\end{align*}
then $f_\varrho$ admits $k$ disjoint sections of self-intersection $-1$. 
Here, $t_{\tilde{\gamma}_i}$ is a Dehn twist mapped to $t_{\gamma_i}$ under $M_g^k \to M_g$. 
Conversely, if a genus-$g$ Lefschetz fibration admits $k$ disjoint sections of self-intersection $-1$, then we obtain such a relation in $M_g^k$. 

Let us now recall the signature formula of Matsumoto and Endo for hyperelliptic Lefschetz fibrations; it will be used to compute the signature of our genus two Lefschetz fibrations obtained in Section~\ref{construction}.  

\begin{thm}[\cite{Ma1},\cite{Ma},\cite{E}]\label{sign} Let $f:X\rightarrow \mathbb{S}^2$ be a genus $g$ hyperelliptic Lefschetz fibration. Let $s_0$ and $s=\Sigma_{h=1}^{[g/2]}s_h$ be the number of non-separating and separating vanishing cycles of $f$, where $s_h$ denotes the number of separating vanishing cycles which separate the surface of genus $g$ into two surfaces, one of which has genus $h$. Then, we have the following formula for the signature 
\begin{eqnarray*}
\sigma(X)=-\frac{g+1}{2g+1}s_0+\sum_{h=1}^{[\frac{g}{2}]}\left(\frac{4h(g-h)}{2g+1}-1\right)s_{h}.
\end{eqnarray*}
\end{thm}

\subsection{Matsumoto's genus two fibration}\label{m}

Yukio Matsumoto's genus two Lefschetz fibration can be conveniently described as the double branched cover of $\mathbb{S}^{2}\times \mathbb{T}^{2}$ with the branch set being the union of two disjoint copies of $\mathbb{S}^{2}\times \{{\rm{pt}}\}$ and two disjoint copies of $\{{\rm{pt}}\} \times \mathbb{T}^{2}$. The resulting branched cover has $4$ singular points, corresponding to the number of intersections of the horizontal spheres and the vertical tori in the branch set. By desingularizing this manifold, one obtains the total space of Matsumoto's fibration, $M = \mathbb{T}^{2}\times \mathbb{S}^{2}\,\#4\CPb$. Notice that the vertical $\mathbb{T}^2$ fibration on $\mathbb{S}^{2}\times \mathbb{T}^{2}$ pulls back to give a fibration of $\mathbb{T}^{2} \times \mathbb{S}^{2}\,\#4\CPb$ over $\mathbb{S}^{2}$. Since a generic fiber of the vertical fibration is the double cover of $\mathbb{T}^2$ branched over $2$ points, it is a genus two surface. Matsumoto proved that \cite{Ma}, the above fibration can be perturbed into Lefschetz one with the global monodromy given by the following word in the mapping class group $M_2$: $(B_{0} B_{1} B_{2} C)^{2} = 1$, where $B_{0}$, $B_{1}$, $B_{2}$, and $C$ denotes the Dehn twists along the curves $\beta_{0}$, $\beta_{1}$, $\beta_{2}$, and $c$ shown in Figure~\ref{fig:matsumoto}.  

Let us denote by $\Sigma_{2}$ the regular fiber of the fibration above, and the images of the standard generators of $\Sigma$ in the fundamental group of $\pi_{1}(M) = \mathbb{Z} \oplus \mathbb{Z}$ as $a_{1}$, $b_{1}$, $a_{2}$, and $b_{2}$. Using the homotopy long exact sequence for a Lefschetz fibration and the existence of $-1$ sphere sections of Matsumoto's fibration, we have the following identification of the fundamental group of $M$ \cite{OS}: 

\begin{equation*}
\pi_{1}(M) = \pi_{1}(\Sigma)/ \langle \beta_{0},\beta_{1},\beta_{2},c\rangle.
\end{equation*}

\begin{eqnarray}
\beta_{0} &=& b_{1}b_{2} ,\, \label{id:b1}\\
c &=& a_{1}b_{1}{a_{1}}^{-1}{b_{1}}^{-1} = a_{2}b_{2}{a_{2}}^{-1}{b_{2}}^{-1} ,\, \label{id:b2}\\
\beta_{1} &=& b_{2}a_{2}{b_{2}}^{-1}a_{1} ,\, \label{id:b3}\\
\beta_{2} &=& b_{2}a_{2}a_{1}b_{1} ,\, \label{id:b4}
\end{eqnarray}

\noindent Hence $\pi_{1}(M) = \ \langle a_{1}, b_{1}, a_{2}, b_{2} \ | \  b_{1}b_{2} = [a_{1}, b_{1}] = [a_{2}, b_{2}] = b_{2}a_{2}{b_{2}}^{-1}a_{1} = 1 \rangle $.

\medskip


\begin{figure}[ht]
\begin{center}
\includegraphics[scale=.63]{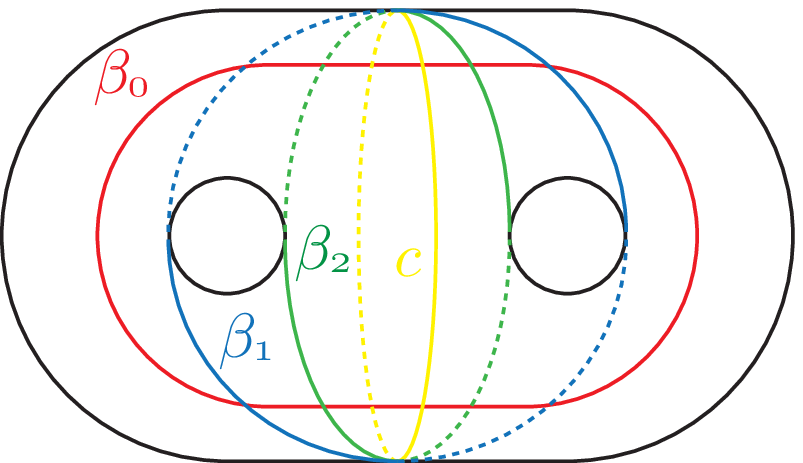}
\caption{Dehn twists for Matsumoto's fibration}
\label{fig:matsumoto}
\end{center}
\end{figure}

\subsection{Xiao's genus two fibrations}\label{x} In this section, we recall a result of Gang Xiao obtained in \cite{Xiao}, where a family of holomorphic genus two Lefschetz fibrations $f(E,d): S(E,d) \rightarrow X(d)$ was constructed for any integer $d \geq 2$. In a special case of $d=3$, the construction in \cite{Xiao}, yields to a genus two Lefschetz fibrations over $\mathbb{S}^{2}$ with $7$ singular fibers, three of which have the separating vanishing cycles. The proof can be derived from Theorem 3.16, Table on pages 52-53, and Theorem 4.5, Example 4.7 on pages 64-66). The case $d=2$ corresponds to Matsumoto's genus two fibration with $8$ singular fibers, that we discussed above. Let us summarize the result obtained by Xiao in the following theorem in a special case of $d=3$, which we will need later on. 

\begin{thm} There exist a holomorphic genus two Lefschetz fibration   \\
$f(E,3): S(E,3) \rightarrow X(3) = \mathbb{S}^{2}$ with $7$ singular fibers, three of which have separating vanishing cycles. Moreover, the complex surface $S(E,3)$ have the following invariants: $p_{g}(S(E,3)) = 0$, $q(S(E,3)) = 1$, $e(S(E,3)) = 3$, $\sigma(S(E,3)) = -3$, $c_1^{2}(S(E,3)) = -3$, and $\chi_{h}(S(E,3)) = 0$.

\end{thm}

Since $p_{g}(S(E,3)) = 0$ and $q(S(E,3)) = 1$, by the Enriques-Kodaira classification of complex surfaces \cite{BHPV} we see that $S(E,3)$ is the blow up of an $\mathbb{S}^{2}$ bundle over $\mathbb{T}^{2}$. Furthermore, using  $\sigma(S(E,3)) = -3$, we deduce that the total space $S(E,3)$ is $\mathbb{T}^{2}\times \mathbb{S}^{2}\,\#3\CPb$. This can be also derived from Proposition 4.1 in \cite{Sato}. 

For the sake of completness and for the beauty of the construction, let us also recall Example 4.7 from \cite{Xiao} below, which gives a geometric description of the fibration above, and provide the details. The discussion below allows us to see the existence of three $(-1)$ sphere sections of Xiao's fibration, and types of its seven singular fibers. We are grateful to R. Pignatelli for explaining this fibrations to us.  

\begin{exmp}\label{XF}(\emph{Xiao's genus two fibration with $p_g=0$ , $q=1$ and $c_1^{2} = -3$}) \\
 In complex projective plane $\mathbb{CP}^2$, let us consider a configuration of six complex lines $L_{1}$, $L_{2}$, $L_{3}$, $L_{1}'$, $L_{2}'$, $L_{3}'$ with $4$ triple points $P_1$, $P_2$, $P_3$, $P_4$, and $3$ double points $P_5$, $P_6$, $P_7$, i.e., a complete quadrangle (see Figure~\ref{fig:qu}). 

\begin{figure}[ht]
\begin{center}
\includegraphics[scale=.63]{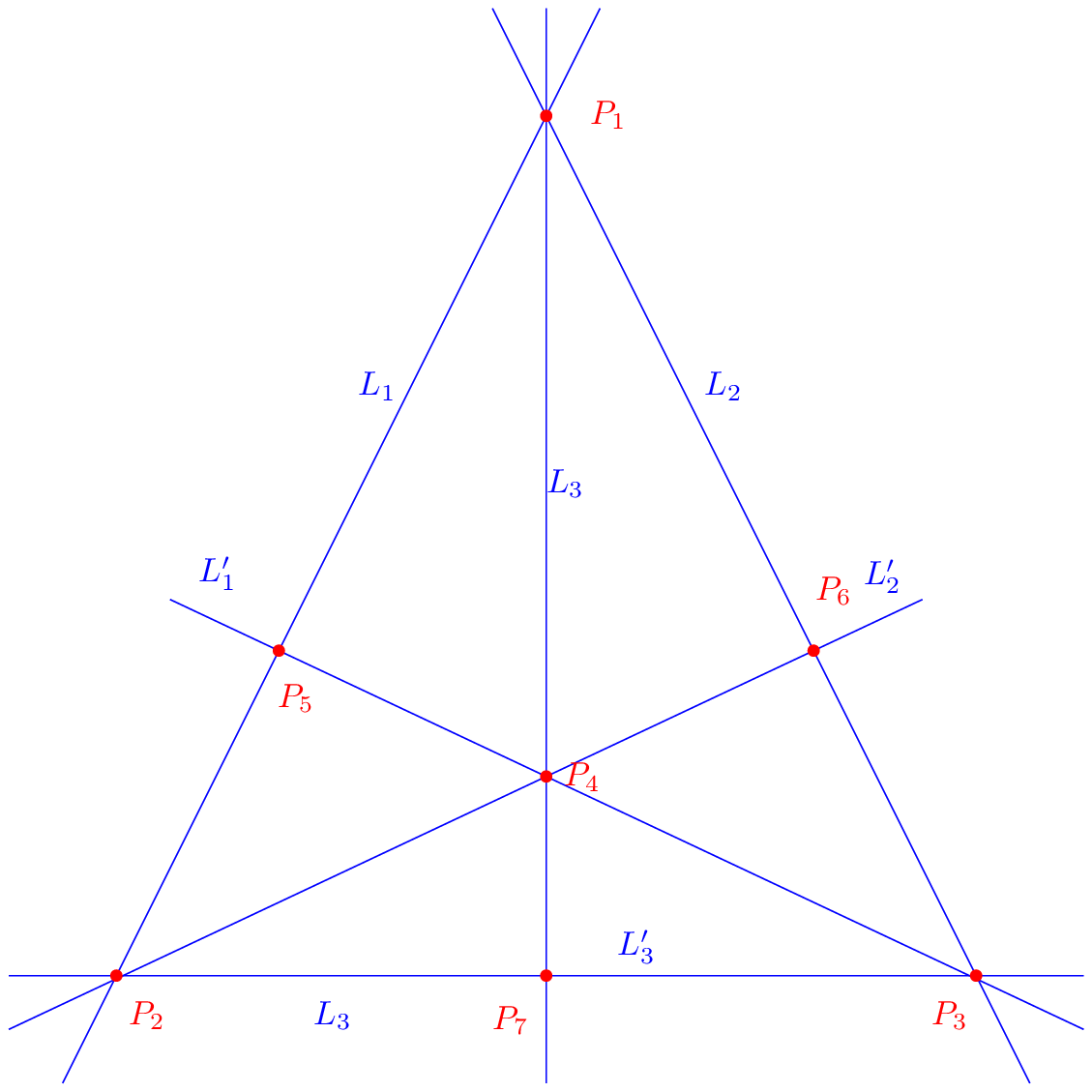}
\caption{Complete quadrangle}
\label{fig:qu}
\end{center}
\end{figure}

Let us choose a generic irreducible complex quartic $C$ passing through these seven points which has the double points at $P_5$, $P_6$, and $P_7$. The existence of such an irreducible quartic is explained in \cite{Xiao}, on page 66. Blow up $\mathbb{CP}^2$ at the points $P_1,\cdots,P_7$, and let $\pi: Y:= \CP\#7\CPb \rightarrow \mathbb{CP}^2$ be the blow up map and $E_i$ be the exceptional divisor corresponding to the blow up at the points $P_i$ for $i=1,\cdots,7$. Let $H$ be the total transform in $Y$ of a line in $\mathbb{CP}^2$, and let $\widetilde{L_{j}}$, $\widetilde{L_{j}'}$, $\widetilde{C}$ be the strict transform of the lines $L_j$, $L_j'$ and quartic $C$ in $\mathbb{CP}^2$. We have

\begin{equation}
\begin{split}
\widetilde{L_1} = H-E_1-E_5-E_2,\, \widetilde{L_1'} = H-E_3-E_4-E_5, \\
\widetilde{L_2} = H- E_1-E_6-E_3,\, \widetilde{L_2'} = H-E_2-E_4-E_6,  \\
\widetilde{L_3} = H- E_1-E_4-E_7,\, \widetilde{L_3'} = H-E_2-E_7-E_3, \\
\widetilde{C} = 4H- E_1- E_2- E_3- E_4 - 2E_5 - 2E_6 - 2E_7.
\label{eq:lines}
\end{split}
\end{equation}

Let
\begin{equation}
D = \widetilde{L_1} + \widetilde{L_2} + \widetilde{L_3}+ \widetilde{L_1'}+ \widetilde{L_2'}+ \widetilde{L_3'} + \widetilde{C} = 10H - 4(E_{1} + \cdots + E_{7})
\label{eq:D}
\end{equation}
be a divisor on $Y$ which has no crossings and divisible by two. The divisor $D$ determines a double cover $p:\widetilde{S} \rightarrow Y$. 

Notice that there is a genus two fibration obtained by pulling back the pencil of lines through point $P_1$. Xiao's fibrations is obtained by considering the relatively minimal model of this fibration. More precisely, he contracts all $(-1)$ spheres contained in fibers. These three $(-1)$ spheres on the double cover arise from the proper transforms of lines $L_1$, $L_2$, and $L_3$. The proper transforms  $\widetilde{L_1'}$,  $\widetilde{L_2'}$,  $\widetilde{L_3'}$ of complex lines $L_1'$, $L_2'$, and $L_3'$ also give rise to the spheres with self-intersection $(-1)$ in the double cover, which are the sections of this fibration. 
 
Let us denote the total space of Xiao's fibration by $S$. Next, we verify that $e(S) = 3$, $\sigma(S) = -3$. By Euler number and Hirzebruch's signature formulas for the double cover, we have  compute the Euler number and the signature of $\widetilde{S}$ as follows.
 \begin{equation*}
\begin{split}
e(\widetilde{S}) = 2e(Y) - e(D)= 2\times 10 - 2 \times 7 = 6, \\
\sigma(\widetilde{S}) = 2\sigma(Y) - \frac{D^2}{2} = 2\times (-6) - \frac{(-12)}{2} = -6
\end{split}
\end{equation*}

Now, using $e(S) = e(\widetilde{S}) - 3= 3$, $\sigma(S) = e(\sigma(\widetilde{S})) + 3 = -3$, we see that $\chi_{h}(S) = 0$ and $c_1^2(S) = -3$. It is shown in \cite{Xiao} that the irregularity $q$ of the complex surface $S$ is $1$. 

One can also see the types of the singular fibers from the above branch cover description. The configurations of curves $\widetilde{L_1} \cup E_{2} \cup E_{5}$, $\widetilde{L_2} \cup E_{3} \cup E_{6}$, and $\widetilde{L_3} \cup E_{4} \cup E_{7}$ arising from the blow up of the points $P_1,\cdots, P_7$, give rise to the three singular fibers with separating vanishing cycles in $S$, once we contracts all $(-1)$ spheres of $\widetilde{S}$ contained in fibers. Also, since the quartic $C$ has exactly four tangent lines passing through the point $P_1$ (this follows from the Riemann-Hurwitz formula), these complex tangent lines in the double cover give us four singular fibers with nonseparating vanishing cycles.  
\end{exmp}

\subsection{Symplectic Kodaira Dimension and Minimality of Symplectic Fiber Sums}\label{kd}

In what follows, we recall the definition of symplectic Kodaira dimension, and state a few theorems on computing the symplectic Kodaira dimension and on symplectic minimality of fiber sums, that are used in proving the main results in the next section. 

The notion of Kodaira dimension has been introduced by K. Kodaira for complex manifolds. It has played a fundamental role in the classification theory of complex surfaces. For symplectic $4$-manifolds, there is also a notion of symplectic Kodaira dimension (see \cite{Li2006, li, McDuffS, LeBrun}). 

\begin{defn} \label{sym Kod}
For a minimal symplectic $4-$manifold $(X^4,\omega)$ with symplectic
canonical class $K_{\omega}$,   the Kodaira dimension of
$(X^4,\omega)$ is defined as follows:

$$
\kappa^s(X^4,\omega)=\begin{cases} \begin{array}{lll}
-\infty & \hbox{ if $K_{\omega}\cdot [\omega]<0$ or} & K_{\omega}\cdot K_{\omega}<0,\\
0& \hbox{ if $K_{\omega}\cdot [\omega]=0$ and} & K_{\omega}\cdot K_{\omega}=0,\\
1& \hbox{ if $K_{\omega}\cdot [\omega]> 0$ and} & K_{\omega}\cdot K_{\omega}=0,\\
2& \hbox{ if $K_{\omega}\cdot [\omega]>0$ and} & K_{\omega}\cdot K_{\omega}>0.\\
\end{array}
\end{cases}
$$

\noindent If $(X^4,\omega)$ is not minimal, then it's Kodaira dimension is defined to be that of any of its minimal models. 
\end{defn}

It is proved in \cite{Li2006} that the symplectic Kodaira dimension is a diffeomorphism invariant. 
Also, it was shown in \cite{DZ} that the symplectic Kodaira dimension coincides with the complex Kodaira dimension when both are defined. 

Let us now recall two theorems proved in \cite{LY} and \cite{Dor}, concerning how the symplectic Kodaira dimension changes under the symplectic fiber sum. These theorems will be used to verify that our Lefschetz fibration have the symplectic Kodaira dimension $2$. 

\begin{thm}\label{kod1} Let $M = X\#_{\Sigma} Y$ be a symplectic fiber sum along an embedded symplectic surface $\Sigma$ in the four-manifolds $X$ and $Y$. Then the symplectic Kodaira dimension is non-decreasing, i.e. $\kappa(M) \geq max\{\kappa(X), \kappa(Y ), \kappa(\Sigma )\}$.
\end{thm}

For genus $0$ case, the above theorem reduces to the following theorem (see \cite{Dor}).

\begin{thm}\label{kod2}Let $M = X\#_{S} Y$ be a symplectic fiber sum on four-manifolds along a symplectic hypersurface $S$ of genus $0$ and $\kappa(X) \geq 0$. Assume there exist no symplectic exceptional spheres disjoint from $S$ in $X$ or $Y$. Then $\kappa(M) \geq \kappa(X)$.
\end{thm}

The minimality of symplectic fiber sums is described by the following theorem. We will use this theorem to verify that the total space of our Lefschetz fibrations are minimal symplectic $4$-manifolds.

\begin{thm}(\cite{U}, \cite{Dor})\label{minimality}. Let $M$ be the symplectic fiber sum $X\#_{V}Y$ of the symplectic manifolds $(X, \omega_{X})$ and $(Y, \omega_{Y})$ along an embedded symplectic surface $V$ of genus $g \geq 0$.
\begin{enumerate}[(i)]
\item The manifold M is not minimal if $X\setminus V_{X}$ or $Y \setminus V_{Y}$ contains an embedded symplectic sphere of self-intersection $-1$ or $X\#_{V} Y = Z\#_{V_{\CP}}\CP$ with $V_{\CP}$ an embedded +4-sphere in class $[V_{\CP}] = 2[H] \in H_{2}(\CP,\mathbb{Z})$ and $Z$ has at least 2 disjoint exceptional spheres $E_i$ each meeting the submanifold $V_{Z} \subset Z$
positively and transversely in a single point with $[E_i] \cdot [V_{X}] = 1$.
\item If $X\#_V Y= Z\#_{V_B}B$ where $B$ is a $\mathbb{S}^2$-bundle over a genus $g$ surface and $V_{B}$ is a section of this bundle then $M$ is minimal if and only if $Z$ is minimal.
\item In all other cases $M$ is minimal.
\end{enumerate}
\end{thm}

\section{Genus two Lefschetz fibration with $b^{+}_{2}=1$ and ${c_1}^{2}=1$}\label{construction}

In this section, we state and prove some our main results, which are Theorems~\ref{thm}, ~\ref{thm1}, and ~\ref{thm2}. We begin with some preparation before stating and proving these theorems. 

Let $\beta_{0}$, $\beta_{1}$, $\beta_{2}$, and $c$ be the vanishing cycles corresponding to the Dehn twists $B_0$, $B_1$, $B_2$, and $C$ of Matsumoto's genus two Lefschetz fibration given in Section~\ref{m}. We will make use of the following two identities

\begin{align}
(B_{0} B_{1} B_{2} C) (B_{0} B_{1} B_{2} C) &= 1 \label{eq:1}\\
B_{0} B_{1} B_{2} C (\beta_{0}) &= \beta_{0} \label{eq:2}
\end{align}

The first identity simply comes from the monodromy of Matsumoto's genus two fibration, and the second identity obtained using the fact that the vertical involution $\iota$ of the genus two surface with two fixed points leaves the curve $\beta_{0}$ invariant. By applying the identities~\eqref{eq:1}, \eqref{eq:2} and part (i) of lemma~\ref{con&braid.lem}, we have 
\begin{align}
1 = B_{0}B_{0} B_{1} B_{2} C B_{1} B_{2} C &= B_{0}^{2}(B_{1}B_{2} C)^{2} = (B_{1}B_{2} C)^{2}B_{0}^{2} \label{eq:3}
\end{align}

Let $f$ and $g$ be the diffeomorphisms of the genus two surface. Let $X(2)$, $X(2, f)$, and  $X(2, f, g)$ denote the total space of the genus two Lefschetz fibration given by the relators $({\theta^{2}})^{2} = (B_0B_1B_2C)^4 = 1$, ${\theta^{2}}{(\theta^{2})}_{f} = \\
(B_0B_1B_2C)^2 f(B_0B_1B_2C)^2)f^{-1} = 1$, and  ${(\theta^{2})}_{f}{(\theta^{2})}_{g} = {\theta^{2}}{(\theta^{2})}_{f^{-1}g} = 1$ in $M_2$, respectively. Notice that $X(2,f)$ and $X(2,f,g)$ obtained via twisted fiber sum of two copies of Matsumoto's fibration, and $X(2, f) = X(2)$ when $[f] = 1 \in M_{2}$.

\begin{thm}\label{thm} There exists a Lefschetz fibration $f_{\theta}: X_{\theta} \rightarrow \mathbb{S}^{2}$ with $b^{+}_{2}(X_{\theta})=1$, ${c_1}^{2}(X_{\theta})=1$ and $\pi_{1}(X_{\theta}) = \mathbb{Z}_{3}$ that can be obtained from twisted fiber sum of two copies of Matsumoto's fibration by applying a single lantern substitution to it's monodromy. Moreover, $X_{\theta}$ is a minimal symplectic $4$-manifold with the symplectic Kodaira dimension $\kappa = 2$.
\end{thm}

\begin{proof} 

Let us consider the diffeomorphism $\phi = C_{4}^{-1} C_{3}^{-1} C_{2}^{-1} C_{1}^{-1}$ of the genus two surface, where $C_{i}$ are Dehn twists diffeomorphism along the standard curves $c_i$ as shown in Figure \ref{fig:iota}. Using Figure~\ref{curves}, we verify the following equations in $\pi_{1}(\Sigma_{2})$.
\begin{figure}[ht]
\begin{center}
\includegraphics[scale=.45]{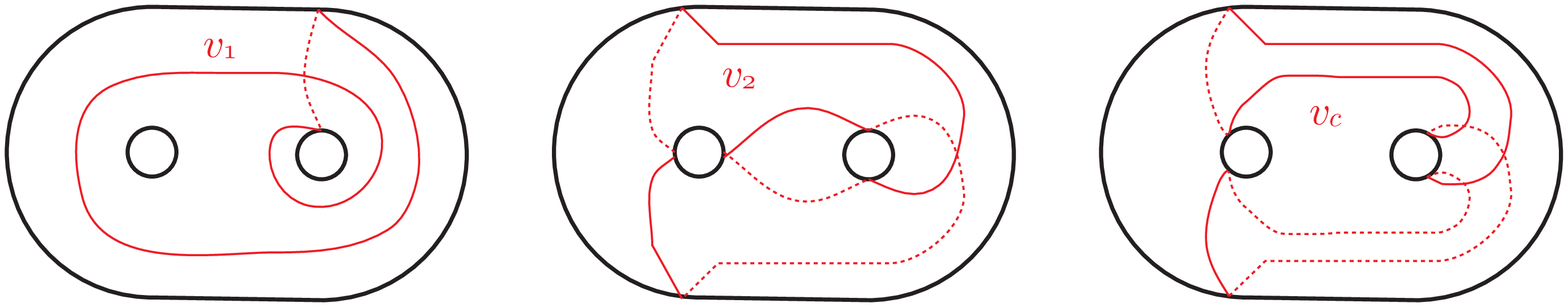}
\caption{Curves $v_1$, $v_2$, $v_c$}
\label{curves}
\end{center}
\end{figure}

\begin{align}
v_{0}:= \phi(\beta_{0}) &= c_{5} = a_{2} \label{eq:4}\\
v_{1}:= \phi(\beta_{1}) &= b_{1}b_{2}a_{2}b_{2} \label{eq:5}\\
v_{2}:= \phi(\beta_{2}) &= b_{1}b_{2}a_{2}a_{1}^{-1}a_{2}b_{2}a_{1} \label{eq:6}\\
v_{c}:= \phi(c) &= b_{1}b_{2}a_{2}a_{1}^{-1}b_{2}^{-1}a_{2}a_{1}^{-1}b_{1}^{-1} \label{eq:7}
\end{align}

\begin{figure}[ht]
\begin{center}
\includegraphics[scale=.63]{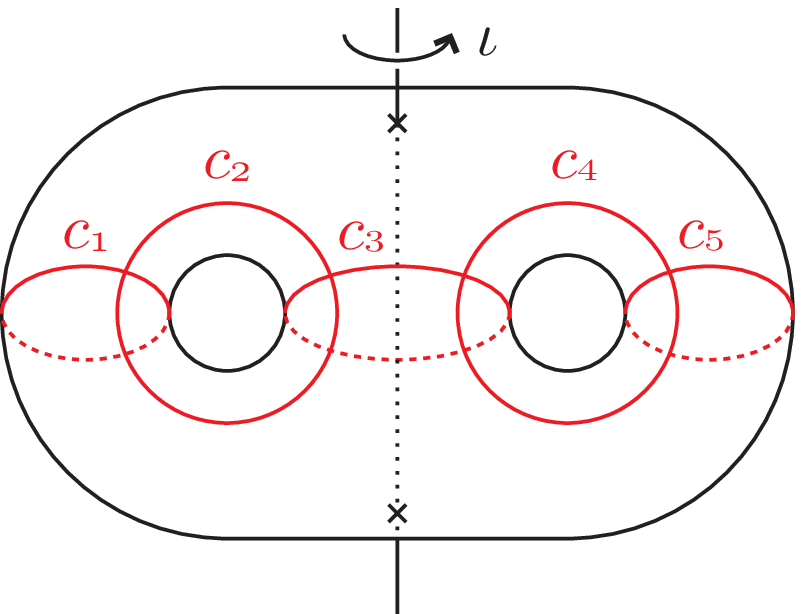}
\caption{Involution $\iota$ with two fixed points}
\label{fig:iota}
\end{center}
\end{figure}

From the description of the map $\iota$, it is easy to check that the followings equation also holds in $\pi_{1}(\Sigma_2)$. 

\begin{align}
w_{0}:= \iota(\phi(\beta_{0})) &= c_{1} = a_{1} \label{eq:8}\\
w_{1}:= \iota(\phi(\beta_{1})) &= b_{2}b_{1}a_{1}b_{1} \label{eq:9}\\
w_{2}:= \iota(\phi(\beta_{2})) &=  b_{2}b_{1}a_{1}a_{2}^{-1}a_{1}b_{1}a_{2}\label{eq:10}\\
w_{c}:= \iota(\phi(c)) &=  b_{2}b_{1}a_{1}a_{2}^{-1}b_{1}^{-1}a_{1}a_{2}^{-1}b_{2}^{-1}\label{eq:11}
\end{align}

By conjugating the global monodromy for Matsumoto's fibration by $\phi$ and $\iota \phi$, applying Lemma~\ref{con&braid.lem} and the discussion above, we obtain the following two words 

\begin{align}
(V_{1} V_{2} V_{c})^{2} (C_{5})^{2} &= 1 \label{eq:12}\\
 (C_{1})^{2}(W_{1} W_{2} W_{c})^{2} &= 1 \label{eq:13}
\end{align}

The Lefschetz fibrations given by above relators are both isomorphic to Matsumoto's fibration, and thus each has a total space $M = \mathbb{T}^{2} \times \mathbb{S}^{2}\,\#4\CPb$. By concatenating the words \eqref{eq:12} and \eqref{eq:13}, we obtain the relation below. Notice that this concatenation in the mapping class group corresponds to the symplectic fiber summing of two copies of Matsumoto's fibration. 

\begin{align}
(V_{1} V_{2} V_{c})^{2} C_{5}^{2} C_{1}^{2}(W_{1} W_{2} W_{c})^{2} &= 1 \label{eq:14}
\end{align}

Finally, by applying the lantern relation $C_{5}^{2} C_{1}^{2} = C_{3}CB_{2}$, we obtain the following relation with $10$ nonseparating vanishing cycles and $5$ separating vanishing cycle.

\begin{align}
\theta = (V_{1} V_{2} V_{c})^{2}C_{3}CB_{2} (W_{1} W_{2} W_{c})^{2} &= 1 \label{eq:15}
\end{align}

Let $X_{\theta}$ denote the total space of the  genus two Lefschetz fibration given by the above monodromy. By applying Lemma~\ref{sign} and Euler characteristic formula for the Lefschetz fibrations, we compute the topological invariants of $X_{\theta}$ as follows:
 \begin{eqnarray*}
e(X_{\theta}) &=& e(\mathbb{S}^{2}) e(\Sigma_{2}) + \# {\emph{singular fibers}} =  2(-2) + 15 =  11 ,\\
\sigma(X_{\theta}) &=& -\frac{3}{5}s_0-\frac{1}{5}s_{1} =  -\frac{3}{5}(10) - \frac{1}{5}(5) = -7, \\
{c_1}^{2}(X_{\theta}) &:=& 2e(X_{\theta}) + 3\sigma(X_{\theta}) = 1, \\
\chi(X_{\theta}) &:=& (e(X_{\theta}) + \sigma(X_{\theta}))/4 = 1
\end{eqnarray*}

Using the Dehn twist factorization $(V_{1} V_{2} V_{c})^{2}C_{3}CB_{2} (W_{1} W_{2} W_{c})^{2} = 1$, let us prove that $\pi_1(X_{\theta})=\mathbb{Z}_{3}$. We have 
$\pi_1(X_{\theta}) = \pi_{1}(\Sigma_{2})/<v_{1}, v_{2}, v_{c}, w_{1}, w_{2}, w_{c}, c_{3}, c, \beta_{2}>$. Thus, the following relations hold in $\pi_1(X_{\theta})$:

\begin{align}
c_{3} &= a_{1}a_{2}^{-1} = 1 \label{eq:25}\\
\beta_{2} &= a_{1}ca_{2} = 1 \label{eq:26}\\
c &= [a_{1},b_{1}] = [a_{2},b_{2}] = 1 \label{eq:28}\\
v_{1} &= b_{1}b_{2}a_{2}b_{2} = 1 \label{eq:29}\\
v_{2} &= b_{1}b_{2}a_{2}a_{1}^{-1}a_{2}b_{2}a_{1} =1 \label{eq:30}\\
v_{c} &= b_{1}b_{2}a_{2}a_{1}^{-1}b_{2}^{-1}a_{2}^{-1}a_{1}b_{1}^{-1} = 1 \label{eq:31}\\
w_{1} &= b_{2}b_{1}a_{1}b_{1} = 1\label{eq:32}\\
w_{2} &=  b_{2}b_{1}a_{1}a_{2}^{-1}a_{1}b_{1}a_{2} = 1\label{eq:33}\\
w_{c} &=  b_{2}b_{1}a_{1}a_{2}^{-1}b_{1}^{-1}a_{1}^-1a_{2}b_{2}^{-1} = 1\label{eq:34}
\end{align}

\noindent where $a_{1}$, $b_{1}$, $a_{2}$, and $b_{2}$ are standard generators for the fundamental group of $\Sigma_2$.
Using the relations \eqref{eq:25}, \eqref{eq:26}, and \eqref{eq:28},  we obtain $a_1=a_2$ and ${a_1}^2=1$. Next, using the relations \eqref{eq:29}, \eqref{eq:30}, and \eqref{eq:31}, we have $a_{1} = a_{2} = 1$ and $b_{1}{b_{2}}^{2} = 1$. Similarly, the relations \eqref{eq:32}, \eqref{eq:33}, and \eqref{eq:34} yields to $b_{2}{b_{1}}^{2} = 1$. Thus, we can conclude that
\begin{align*}
\pi_1(X_{\theta}) & =\,  <b_1, b_2 \,| \, b_{1}{b_{2}}^{2}, \, b_{2}{b_{1}}^{2}> \,=\, <b_1\,|\, {b_{1}}^{3}> \, = \, \mathbb{Z}_{3} 
\end{align*}

Since we have verified that $b_{1}(X_{\theta}) = 0$, we compute $b_{2}^{+}(X_{\theta}) =  2\chi(X_{\theta}) + 2b_1(X_{\theta}) -1 = 1$.

To show that $X_{\theta}$ is symplectically minimal, let us recall that $X_{\theta}$ was obtained from $X(2, \phi, \iota \phi)$ via a single lantern substitution (i.e. by a rational blowdown surgery along $-4$ sphere (see \cite{FS1} for the definition of rational blowdown). Moreover, the rational blowdown surgery along $-4$ sphere can be viewed as the symplectic sum: we have $X_{\theta} = X(2,\phi,\iota \phi)\#_{S, V_{\CP}}\CP$, where $V_{\CP}$ an embedded +4-sphere in class of $[V_{\CP}] = 2[H] \in H_{2}(\CP,\mathbb{Z})$, and $S$ is a symplectic $-4$ sphere in $X(2\phi, \iota \phi)$. Since $X(2,\phi, \iota \phi) = M\#_{\Sigma_{2}}M$ is symplectically minimal by Theorem~\ref{minimality} (case (iii) applies), it also follows from Theorem~\ref{minimality} that $X_{\theta}$ is a minimal symplectic $4$-manifold. Finally, it follows from Theorem~\ref{kod1}, and the above symplectic fiber sum decomposition that $\kappa(X(2,\phi, \iota \phi)) \geq \kappa(X(2)) \geq max\{\kappa(M), \kappa(\Sigma_{2} )\} = 1$ and consequently, we have
$\kappa(X_{\theta}) \geq \kappa(X(2,\phi, \iota \phi)) \geq max\{\kappa(X(2)), \kappa(\CP), \kappa(\mathbb{S}^{2} )\} \geq1.$ Since  $X_{\theta}$ is symplectically minimal and ${c_1}^{2}(X_{\theta})=1$, we conclude that  $\kappa(X_{\theta})=2$. \end{proof}

The proofs of the following theorem is similar to the proof of Theorem~\ref{thm}. We will provide the details of the fundamental group computations, but omit the proofs of minimality and the  symplectic Kodaira dimension computations for the sake of brevity, which are identical to as in Theorem~\ref{thm}. 

\begin{thm}\label{thm1} There exists a genus two Lefschetz fibration $f_{\eta}: X_{\eta} \rightarrow \mathbb{S}^{2}$ with $b^{+}_{2}(X_{\eta})=1$, ${c_1}^{2}(X_{\eta})=1$ and $\pi_{1}(X_{\eta}) = \mathbb{Z}_{4}$ that can be obtained from twisted fiber sum of two copies of Matsumoto's fibration by applying a single lantern substitution to it's monodromy. Moreover, $X_{\eta}$ is minimal symplectic $4$-manifold with the symplectic Kodaira dimension $\kappa(X_{\eta}) = 2$.
\end{thm}

\begin{proof} We consider the diffeomorphism $\psi = C_{5}C_{3}^{-1}C_{2}^{-1}C_{1}^{-1}$ of the genus two surface. Using Figure~\ref{curves}, it easy to verify the following equations hold in $\pi_{1}(\Sigma_{2})$.

\begin{align}
x_{0}:= \psi(\beta_{0}) &= c_{4}=b_2 \label{eq1:4}\\
x_{1}:= \psi(\beta_{1}) &= b_{1}b_{2}a_{2}a_{2} \label{eq1:5}\\
x_{2}:= \psi(\beta_{2}) &= b_{1}b_{2}a_{2}{b_2}^{-1}{a_1}^{-1}a_{2}a_{1} \label{eq1:6}\\
x_{c}:= \psi(c) &= {a_1}^{-1}{a_2}^{-1}a_{1}b_{2}a_{2}{b_2}^{-1} \label{eq1:7}\\
a:= C_5^{-1}C_4^{-1}(\beta_{2}) &= a_{1}b_{2} \label{eq1:8}\\
b:= C_5^{-1}C_4^{-1}(c_{3}) &= b_{1}b_{2}{b_1}^{-1}{a_1}^{-1} \label{eq1:9}
\end{align}

\begin{figure}[ht]
\begin{center}
\includegraphics[scale=.45]{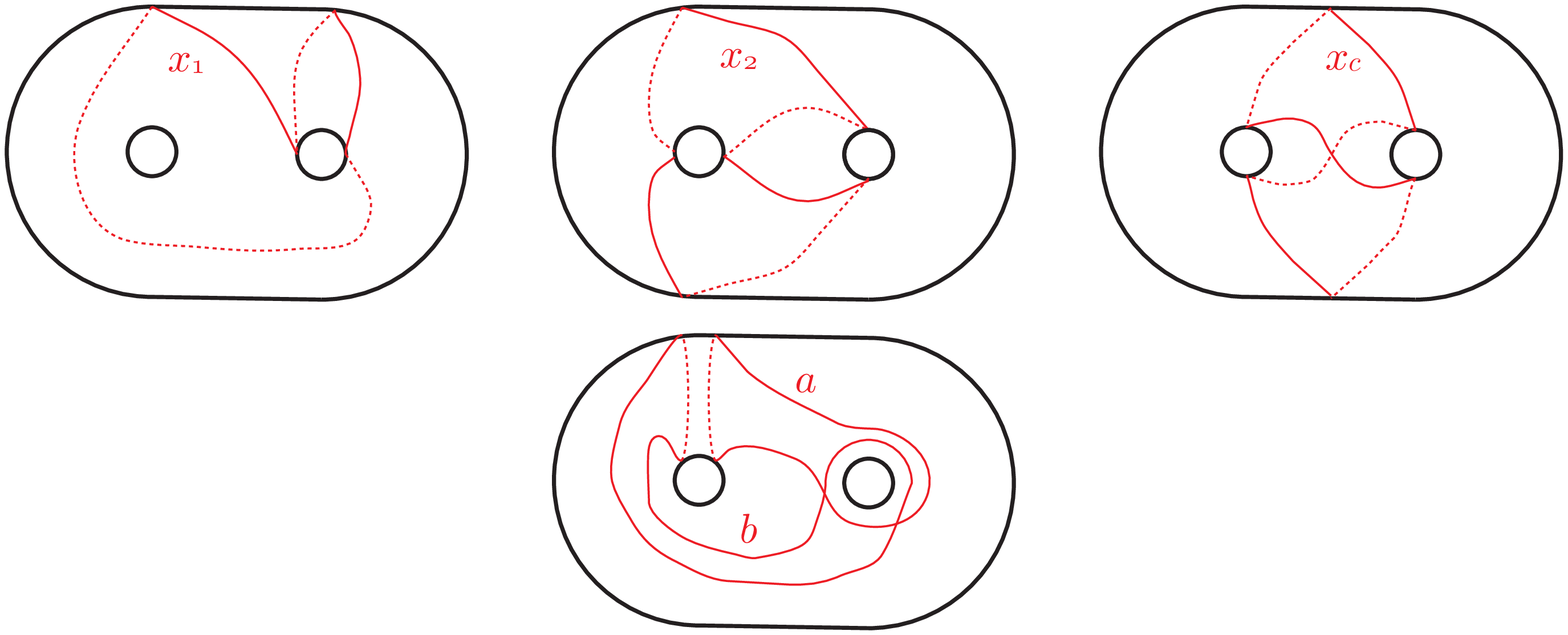}
\caption{Curves $x_1$, $x_2$, $x_c$, $a$ and $b$}
\label{curves}
\end{center}
\end{figure}



We conjugating the global monodromy of Matsumoto's fibration by $\psi$ and apply Lemma~\ref{con&braid.lem} to obtain the following word 
\begin{align}
(X_{1} X_{2} X_{c})^{2} (C_{4})^{2} &= 1 \label{eq2:1}
\end{align}

Next, we concatenate the words \eqref{eq:13} and \eqref{eq2:1} to obtain the relation given below. 
\begin{align}
(X_{1} X_{2} X_{c})^{2} C_{4}^{2} C_{1}^{2}(W_{1} W_{2} W_{c})^{2} &= 1 \label{eq2:2}
\end{align}

Finally, by applying the lantern relation $C_{4}^{2} C_{1}^{2} = B C A$, which is obtained by conjugating the lantern relation $C_{5}^{2} C_{1}^{2} = C_{3}CB_{2}$ by $C_{5}^{-1}C_4^{-1}$ and the equations (\ref{eq1:8}) and (\ref{eq1:9}), we obtain the following relation which has $10$ nonseparating vanishing cycles and $5$ separating vanishing cycle.

\begin{align}
\eta = (X_{1} X_{2} X_{c})^{2} B C A (W_{1} W_{2} W_{c})^{2} &= 1 \label{eq2:3}
\end{align}

Let $X_{\eta}$ denote the total space of the  genus two Lefschetz fibration given by the above monodromy. Next, let us compute the topological invariants of $X_{\eta}$ as given below:
 \begin{eqnarray*}
e(X_{\eta}) &=& e =  2(-2) + 15 =  11 ,\\
\sigma(X_{\eta}) &=& -\frac{3}{5}s_0-\frac{1}{5}s_{1} =  -\frac{3}{5}(10) - \frac{1}{5}(5) = -7, \\
{c_1}^{2}(X_{\eta}) &:=& 2e(X_{\eta}) + 3\sigma(X_{\eta}) = 1, \\
\chi(X_{\eta}) &:=& (e(X_{\eta}) + \sigma(X_{\eta}))/4 = 1
\end{eqnarray*}

\noindent In what follows, we use the Dehn twist factorization $(X_{1} X_{2} X_{c})^{2} B C A (W_{1} W_{2} W_{c})^{2} = 1$, to prove that $\pi_1(X_{\eta})=\mathbb{Z}_{4}$. We have 
$\pi_1(X_{\eta}) = \pi_{1}(\Sigma_{2})/<x_1, x_2, x_c, w_{1}, w_{2}, w_{c}, a, b, c>$. Thus, the following relations hold in $\pi_1(X_{\eta})$:

\begin{align}
a &= a_{1}b_{2} = 1 \label{eq2:4}\\
b &= b_{1}b_{2}{b_1}^{-1}{a_1}^{-1} = 1 \label{eq2:5}\\
c &= [a_{1},b_{1}] = 1 \label{eq2:6}\\
x_{1} &= b_{1}b_{2}a_{2}a_{2} \label{eq2:7}\\
x_{2} &= b_{1}b_{2}a_{2}{b_2}^{-1}{a_1}^{-1}a_{2}a_{1} \label{eq2:8}\\
x_{c} &= {a_1}^{-1}{a_2}^{-1}a_{1}b_{2}a_{2}{b_2}^{-1} \label{eq2:9}\\
w_{1} &= b_{2}b_{1}a_{1}b_{1} = 1\label{eq2:10}\\
w_{2} &=  b_{2}b_{1}a_{1}a_{2}^{-1}a_{1}b_{1}a_{2} = 1\label{eq2:11}\\
w_{c} &=  b_{2}b_{1}a_{1}a_{2}^{-1}b_{1}^{-1}a_{1}^-1a_{2}b_{2}^{-1} = 1\label{eq2:12}
\end{align}

\noindent where $a_{1}$, $b_{1}$, $a_{2}$, and $b_{2}$ are standard generators for the fundamental group of $\Sigma_2$. Using the relations \eqref{eq2:4}, \eqref{eq2:5}, and \eqref{eq2:6}, we obtain $b_2=a_1^{-1}$ and ${a_1}^2=1$. Next, using the relations \eqref{eq2:7}, \eqref{eq2:8}, and \eqref{eq2:9}, we have $a_{1} = b_{2} = 1$ and $b_{1}{a_{2}}^{2} = 1$. Similarly, the relations \eqref{eq2:10}, \eqref{eq2:11}, and \eqref{eq2:12} yield to ${b_{1}}^{2} = 1$. Thus, we can conclude that
\begin{align*}
\pi_1(X_{\eta}) & =\,  <a_2, b_1 \,| \, b_{1}{a_{2}}^{2}, \, {b_{1}}^{2}> \,=\, <a_2\,|\, {a_{2}}^{4}> \, = \, \mathbb{Z}_{4} 
\end{align*}

Since $b_{1}(X_{\eta}) = 0$, we compute $b_{2}^{+}(X_{\eta}) =  2\chi(X_{\eta}) + 2b_1(X_{\eta}) -1 = 1$. \end{proof}

\begin{thm}\label{thm2} Let $n$ be any integer. There exists a family of Lefschetz fibration $f_{n}: X_{\theta_n} \rightarrow \mathbb{S}^{2}$ with $e(X_{\theta_n})=11$, $b^{+}_{2}(X_{\theta_n})=1$, ${c_1}^{2}(X_{\theta_n})=1$, and

\begin{enumerate}[(a)]
\item $\pi_{1}(X_{\theta_n}) = 1$ for $n = -1, -3$
\item $\pi_{1}(X_{\theta_n}) = \mathbb{Z}$ for $n=-2$
\item $\pi_{1}(X_{\theta_n}) = \mathbb{Z}_{|n+2|}$ for $n \neq -3, -2, -1$
\end{enumerate}

\noindent that can be obtained from twisted fiber sum of two copies of Matsumoto's fibration by applying a single lantern substitution to it's monodromy. Moreover, $X_{\theta_n}$ is minimal symplectic $4$-manifold with the symplectic Kodaira dimension $\kappa(X_{\theta_n}) = 2$.
\end{thm}

\begin{proof}  Let us consider the diffeomorphisms $\psi_n = C_{1}^{n} C_{5}^{-1} C_{4}^{-1} C_{3}^{-1} C_{2}^{-1} C_{1}^{-1}$ and $\phi_{n} = C_1^{n} C_4$ of the genus two surface, where $C_{i}$ are Dehn twists diffeomorphism along the standard curves $c_i$ as shown in Figure \ref{fig:iota}. Using Figure~\ref{curves}, it is easy to verify the following equations in $\pi_{1}(\Sigma_{2})$.

\begin{figure}[ht]
\begin{center}
\includegraphics[scale=.45]{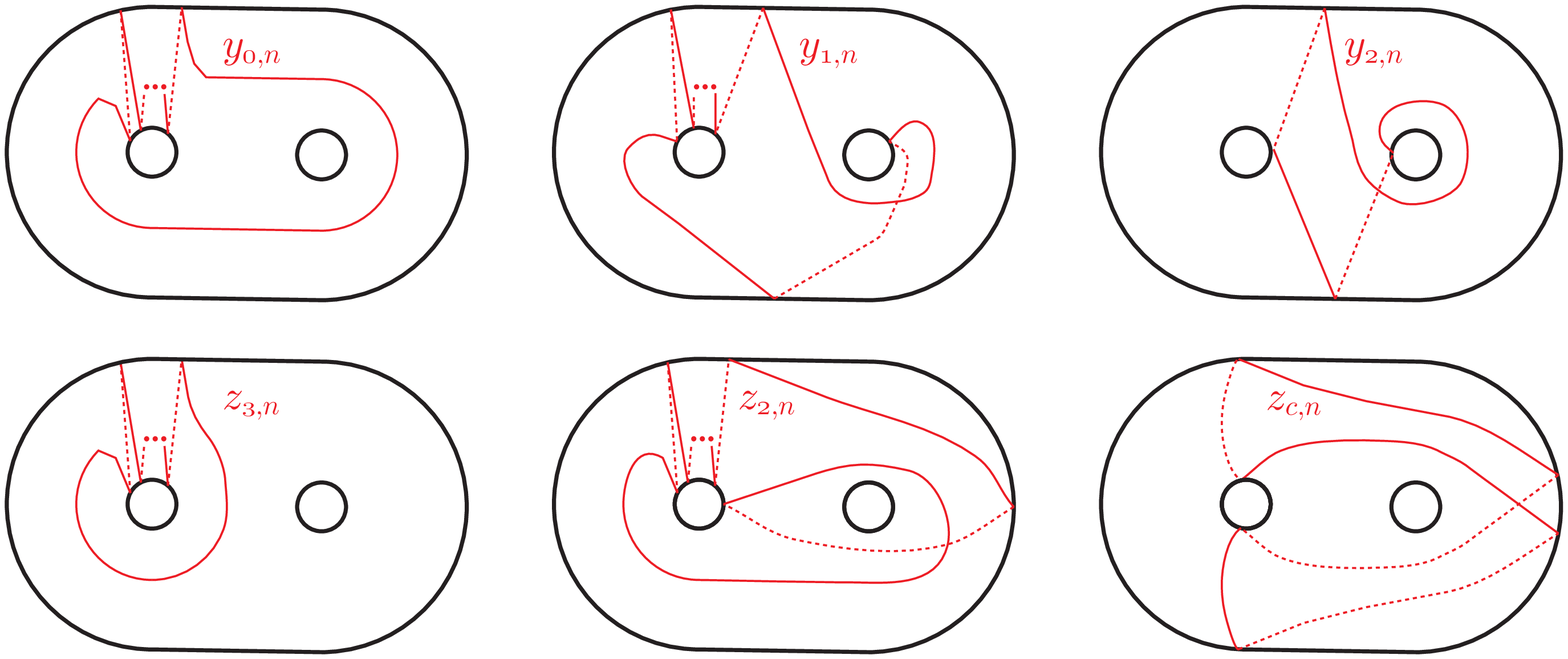}
\caption{Curves $y_{0,n}$, $y_{1,n}$, $y_{2,n}$, $z_{3,n}$, $z_{2,n}$ and $z_{c,n}$}
\label{fig:iota}
\end{center}
\end{figure}

\begin{align}
y_{0,n}:= \phi_{n}(\beta_{0}) &= \psi_{n}(c_1) = b_{1}b_{2}{a_1}^n \label{eq2:13}\\
y_{1,n}:= \phi_{n}(\beta_{1}) &= b_{1}b_{2}a_{2}b_{2}^{-1}{a_1}^{n+1} \label{eq2:14}\\
y_{2,n}:= \phi_{n}(\beta_{2}) &= b_{2}{a_{2}}^{-1}{c}^{-1}{a_{1}}^{-1} \label{eq2:15}\\
y_{c,n}:= \phi_{n}(c) &= c = [a_{1},b_{1}] \label{eq2:16}\\
z_{5,n}:= \psi_{n}(c_5) &= c_4 = b_2 \label{eq2:17}\\
z_{3,n}:= \psi_{n}(c_3) &= b_1{a_1}^n \label{eq2:18}\\
z_{2,n}:= \psi_{n}(\beta_2) &= b_{1}b_{2}{a_1}^{n+1}b_{1}b_{2}{b_{1}}^{-1}{a_{1}}^{-1} \label{eq2:19}\\
z_{c,n}:= \psi_{n}(c) &= b_{1}b_{2}a_{1}{b_{2}}^{-1}{b_{1}}^{-1}{a_{1}}^{-1} \label{eq2:20}
\end{align}




By conjugating the global monodromy for Matsumoto's fibration by $\phi_n$, applying Lemma~\ref{con&braid.lem}, we obtain the following word 
\begin{align}
(Y_{0,n})^{2} (Y_{1,n} Y_{2,n} C)^{2} &= 1 \label{eq2:21}
\end{align}

Since the Lefschetz fibrations given by the relators (\ref{eq2:1}) and (\ref{eq2:21}) are both isomorphic to Matsumoto's fibration, each has a total space $M = \mathbb{T}^{2} \times \mathbb{S}^{2}\,\#4\CPb$. Next, we concatenate the words \eqref{eq2:1} and \eqref{eq2:21} to obtain the relation below. 
\begin{align}
(X_{1} X_{2} X_{c})^{2} C_{4}^{2} Y_{0,n}^{2} (Y_{1,n} Y_{2,n} C)^{2} &= 1 \label{eq2:22}
\end{align}

Finally, by applying the lantern relation $C_{4}^{2} Y_{0,n}^{2} = Z_{3,n} Z_{c,n} Z_{2,n}$, which is obtained by conjugating the lantern relation $C_{5}^{2} C_{1}^{2} = C_{3}CB_{2}$ by $\psi_n$ and the equations (\ref{eq2:13}), (\ref{eq2:17})--(\ref{eq2:20}), we obtain the following relation with $10$ nonseparating vanishing cycles and $5$ separating vanishing cycle.

\begin{align}
\theta_n = (X_{1} X_{2} X_{c})^{2} Z_{3,n} Z_{c,n} Z_{2,n} (Y_{1,n} Y_{2,n} C)^{2} &= 1 \label{eq2:23}
\end{align}

Let $X_{\theta_n}$ denote the total space of the genus two Lefschetz fibration given by the monodromy $\theta_n =1$. Applying Lemma~\ref{sign} and Euler characteristic formula for the Lefschetz fibrations, we compute the topological invariants of $X_{\theta_n}$ as follows:
 \begin{eqnarray*}
e(X_{\theta_n}) &=& 2(-2) + 15 =  11 ,\\
\sigma(X_{\theta_n}) &=& -\frac{3}{5}s_0-\frac{1}{5}s_{1} =  -\frac{3}{5}(10) - \frac{1}{5}(5) = -7, \\
{c_1}^{2}(X_{\theta_n}) &:=& 2e(X_{\theta_n}) + 3\sigma(X_{\theta_n}) = 1, \\
\chi(X_{\theta_n}) &:=& (e(X_{\theta_n}) + \sigma(X_{\theta_n}))/4 = 1
\end{eqnarray*}

\noindent To compute $\pi_1(X_{\theta_n})$, we use word
$(X_{1} X_{2} X_{c})^{2} Z_{3,n} Z_{c,n} Z_{2,n} (Y_{1,n} Y_{2,n} C)^{2} = 1$. We have 
\begin{align*}
\pi_1(X_{\theta_n}) = \pi_{1}(\Sigma_{2})/<y_{1,n}, y_{2,n}, c, z_{3,n}, z_{c,n}, z_{2,n}, x_{1}, x_{2}, x_{c}>.
\end{align*}
Thus, the following relations hold in $\pi_1(X_{\theta_n})$:

\begin{align}
z_{3,n} &= b_1{a_1}^n = 1 \label{eq2:24}\\
z_{2,n} &= b_{1}b_{2}{a_1}^{n+1}b_{1}b_{2}{b_{1}}^{-1}{a_{1}}^{-1} = 1 \label{eq2:25}\\
z_{c,n} &= b_{1}b_{2}a_{1}{b_{2}}^{-1}{b_{1}}^{-1}{a_{1}}^{-1} = 1 \label{eq2:26}\\
y_{1,n} &= b_{1}b_{2}a_{2}{b_{2}}^{-1}{a_1}^{n+1} = 1 \label{eq2:27}\\
y_{2,n} &= b_{2}{a_{2}}^{-1}{c}^{-1}{a_{1}}^{-1} = 1 \label{eq2:28}\\
c &= [a_{1},b_{1}] = 1 \label{eq2:29}\\
x_{1} &= b_{1}b_{2}a_{2}a_{2} = 1 \label{eq2:30}\\
x_{2} &= b_{1}b_{2}a_{2}{b_2}^{-1}{a_1}^{-1}a_{2}a_{1} = 1 \label{eq2:31}\\
x_{c} &= {a_1}^{-1}{a_2}^{-1}a_{1}b_{2}a_{2}{b_2}^{-1} = 1 \label{eq2:32}
\end{align}

\noindent where $a_{1}$, $b_{1}$, $a_{2}$, and $b_{2}$ are standard generators for the fundamental group of $\Sigma_2$.
Using the relations \eqref{eq2:24}, \eqref{eq2:25}, and \eqref{eq2:26},  we obtain ${a_1}^{n}b_{1}=1$ and ${b_2}^2=1$. Next, using the relations \eqref{eq2:27}, \eqref{eq2:28}, and \eqref{eq2:29}, we have $b_{2} = 1$ and $a_{1}a_{2} = 1$. Similarly, the relations \eqref{eq:32}, \eqref{eq:33}, and \eqref{eq:34} yields to $b_{1}{a_{1}}^{-2} = 1$. Thus, we can conclude that
\begin{align*}
\pi_1(X_{\theta_n}) & =\,  <a_{1}, a_{2}, b_{1} \,| \, a_{1}^{n}b_{1}, \, a_{1}a_{2}, \, b_{1}{a_{1}}^{-2}> \,=\, <a_1\,|\, {a_{1}}^{n+2}> \, 
\end{align*}

Notice that when $n= -3$ or $-1$, the group given by the above presentation is trivial. If $n = -2$, then $\pi_1(X_{\theta_n}) = \, \mathbb{Z}$, which is generated by $a_1$. When $n \neq -3, -2, -1$, then $\pi_1(X_{\theta_n}) = \, \mathbb{Z}_{|n+2|}$ non-trivial finite cyclic group of order $|n+2|$.
Since we have verified that $b_{1}(X_{\theta_n}) = 0$ for $n \neq -2$, we compute $b_{2}^{+}(X_{\theta_n}) =  2\chi(X_{\theta_n}) + 2b_1(X_{\theta_n}) -1 = 1$. For $n=-2$, we have $b_{1}(X_{\theta_n}) = 1$ and $b_{2}^{+}(X_{\theta_n}) =  2\chi(X_{\theta_n}) + 2b_1(X_{\theta_n}) -1 = 3$. Thus, $X_{\theta_{-2}}$ is non K\"{a}hler symplectic $4$-manifold.

To show $X_{\theta_n}$ is symplectically minimal, we proceed as in the proof of Theorem~\ref{thm}. We will use the facts that $X_{\theta_n}$ was obtained from $X(2, \phi_n, \psi_n)$ by a single lantern substitution, $X_{\theta_n} = X(2,\phi_n, \psi_n)\#_{S, V_{\CP}}\CP$ and $X(2, \phi_n, \psi_n)$ is symplectically minimal. Using Theorem~\ref{minimality}, we deduce that $X_{\theta_n}$ is a minimal symplectic $4$-manifold. Finally, by applying Theorem~\ref{kod1}, and the above symplectic fiber sum decomposition, we have $\kappa(X(2,\phi_n, \psi_n)) \geq \kappa(X(2)) \geq max\{\kappa(M), \kappa(\Sigma_{2} )\} = 1$ and consequently, we have
$\kappa(X_{\theta_n}) \geq \kappa(X(2,\phi_n, \psi_n)) \geq max\{\kappa(X(2)), \kappa(\CP), \kappa(\mathbb{S}^{2} )\} \geq1.$ Since  $X_{\theta}$ is symplectically minimal and ${c_1}^{2}(X_{\theta_n})=1$, we have  $\kappa(X_{\theta_n})=2$. When $n=-1$ or $n=-3$, using $\kappa(X_{\theta_n})=2$ we can deduce that $X_{\theta_{-1}}$ and $X_{\theta_{-3}}$ are exotic symplectic copies of $\CP\#8\CPb$.
\end{proof}

\begin{rmk}\label{othercases} Further examples can be obtained using other lantern substitutions. We present a few additional cases here, but omit the details. Let choose the diffeomorphism $\phi=C_4C_3C_5C_4$. We have $\phi(c_5)=c_3$, and we set $u_1:=\phi(v_1)$, $u_2:=\phi(v_2)$ and $u_c:=\phi(v_c)$. Therefore, we obtain the relations $(U_{1}U_{2}U_{c})^2C_3^2(B_0)^{2}(B_{1} B_{2} B_{c})^{2}=1$ and $(U_{1}U_{2}U_{c})^2C_3^2(C_{1})^{2}(W_{1} W_{2} W_{c})^{2}=1$. Notice that we can apply lantern substitutions to these relations since $c_3$ is disjoint from $\beta_0$ and $c_1$. We have verified that the fundamental groups of the resulting Lefschetz fibrations are $\mathbb{Z}_2$ and $\mathbb{Z}_3$, respectively.
 \end{rmk}

\begin{rmk}\label{decomp} It is an interesting problem to determine if the Lefschetz fibrations $f_{\theta}: X_{\theta} \rightarrow \mathbb{S}^2$, $f_{\eta}: X_{\eta} \rightarrow \mathbb{S}^2$, and $f_{n}: X_{\theta_n} \rightarrow \mathbb{S}^2$, which all have the minimal total spaces, are fiber sum indecomposible or not. Notice that if the total spaces $X_{\theta}$, $X_{\eta}$, and $X_{\theta_n}$, which all have $n = 10$ nonseparating and $s = 5$ separating vanishing cycles, decomposes as a fiber sum $X_1 \#_{\Sigma_2} X_2$ of two Lefschetz fibrations over $\mathbb{S}^2$, then we can assume that $X_1$ has $2$ separating and $X_2$ has $3$ separating vanishing cycles. The proof rests on the following three facts: (1) there exist no Lefschetz fibration over $\mathbb{S}^2$ with only separating vanishing cycles; (2) for a genus two Lefschetz fibration over $\mathbb{S}^2$ with $n$ nonseparating and $s$ separating vanishing cycles, $(n + 2s) \equiv 0 \mod{10}$; (3) there exist no Lefschetz fibration over $\mathbb{S}^2$ with $n=2$, and $s=4$. 
Using the number of vanishing cycles and their type for $X_1$ ($n_{1}=6$ and $s_{1}=2$) and $X_2$ ($n_{2}=4$ and $s_{2}=3$), we compute their topological invariants: $e(X_1) = 4$,  $\sigma(X_1) = -4$,  $\chi(X_1) = 0$,  $c_1^2(X_1) = -4$, and $e(X_2) = 3$,  $\sigma(X_2) = -3$,  $\chi(X_2) = 0$,  $c_1^2(X_2) = -3$. Consequently, it follows from Proposition 4.1 in \cite{Sato}
that $X_1$ and $X_2$ must be diffeomorphic to $\mathbb{T}^{2}\times \mathbb{S}^{2}\,\#4\CPb$ and $\mathbb{T}^{2}\times \mathbb{S}^{2}\,\#3\CPb$, respectively. It is possibile that our fibration decomposes as the fiber sum of Matsumoto's fibration \cite{Ma1} (with $n = 6$ and $s = 2$) and Xiao's fibration \cite{Xiao}  (with $n=4$ and $s = 3$). However, the explicit monodromy of latter fibration is unknown. A few weeks after we announced our result (see the earlier versions of this preprint in arXiv), I. Baykur and M. Korkmaz posted a preprint \cite{BK} which constructs explicit monodromy of such a genus two Lefschetz fibration. It is not clear whether their fibration is related to Xiao's fibration. One can compute the monodromy of Xiao's fibration using the branched cover description outlined in Example~\ref{XF}, and applying the braid monodromy techniques of B. Moishezon. Another, interesting problem would be to determine if our fibrations $f_{\theta}: X_{\theta} \rightarrow \mathbb{S}^2$, $f_{\eta}: X_{\eta} \rightarrow \mathbb{S}^2$, and $f_{n}: X_{\theta_n} \rightarrow \mathbb{S}^2$ (for $n \neq -2$ and $-7 \leq n \leq 3$) are holomorphic or not. It is known that our fibration $f_{n}: X_{\theta_n} \rightarrow \mathbb{S}^2$ can not be holomorphic for all other $n$. In a follow-up project, we will study these problems. 

\end{rmk}

\section{The existence of a section}\label{section}
First, we show that our example in Theorem~\ref{thm} has a sphere section of square $-2$. 
Let $\widetilde{B}_0$, $\widetilde{B^\prime}_0$, $\widetilde{B}_1$, $\widetilde{B}_2$ and $\widetilde{C}$ be the Dehn twists along the simple closed curve $\widetilde{\beta}_{0}$, $\widetilde{\beta^\prime}_{0}$, $\widetilde{\beta}_{1}$, $\widetilde{\beta}_{2}$, and $\widetilde{c}$ as in Figure~\ref{fig:matsumotosc}. 
\begin{figure}[ht]
\begin{center}
\includegraphics[scale=.63]{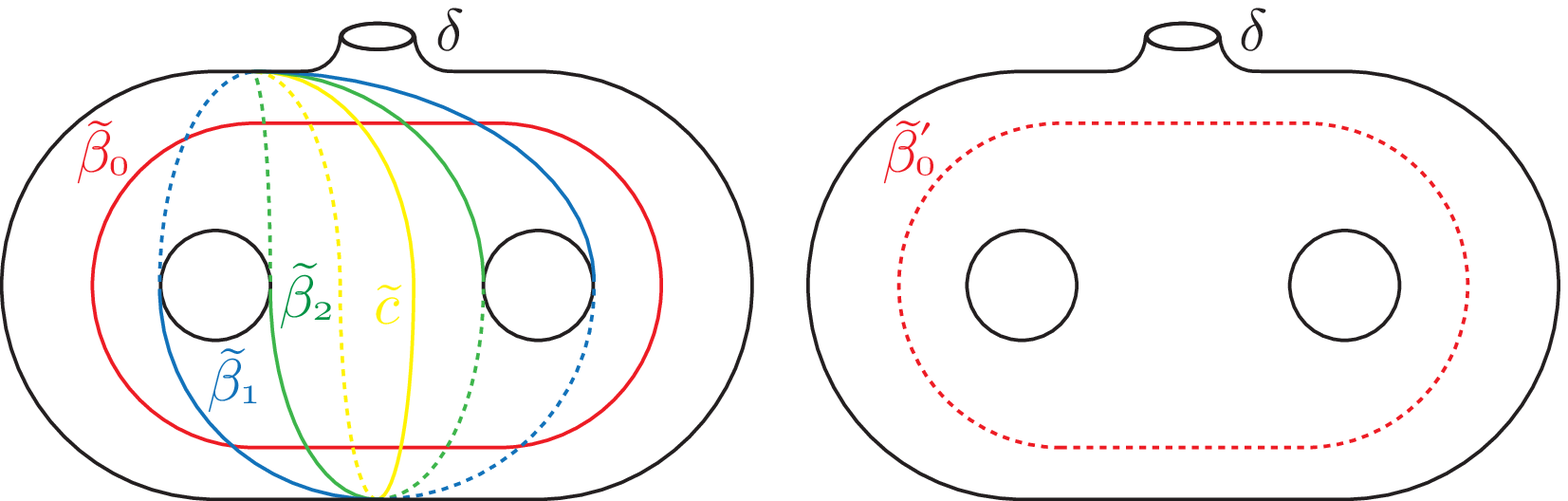}
\caption{The curves $\widetilde{\beta}_{0}$, $\widetilde{\beta^\prime}_{0}$, $\widetilde{\beta}_{1}$, $\widetilde{\beta}_{2}$, and $\widetilde{c}$}
\label{fig:matsumotosc}
\end{center}
\end{figure}
Ozbagci and Stipsicz \cite{OS} show that there is the following lift of $(B_0B_1B_2C)^2=1$ in $M_2$ to $M_2^1$:
\begin{align}\label{matsumotolift}
(\widetilde{B}_0 \widetilde{B}_1 \widetilde{B}_2 \widetilde{C})^2 = \partial,
\end{align}
where $\partial$ is the Dehn twist along the simple closed curve $\delta$ parallel to the boundary component of $\Sigma_2^1$. This guarantees that Matsumoto's fibration admits a sphere section of square $-1$. By drawing corresponding curves and applying the corresponding Dehn twist, we obtain the following identity. 
\begin{align*}
\widetilde{B}_0 \widetilde{B}_1 \widetilde{B}_2 \widetilde{C}(\widetilde{\beta}_0) = \widetilde{\beta^\prime}_0.
\end{align*}
Define simple closed curves to be 
\begin{align*}
&\widetilde{\beta^\prime}_1 := (\widetilde{B}_0 \widetilde{B}_1 \widetilde{B}_2 \widetilde{C})^{-1}(\widetilde{\beta}_1),& 
&\widetilde{\beta^\prime}_2 := (\widetilde{B}_0 \widetilde{B}_1 \widetilde{B}_2 \widetilde{C})^{-1}(\widetilde{\beta}_2),& 
&\widetilde{c^\prime} := (\widetilde{B}_0 \widetilde{B}_1 \widetilde{B}_2 \widetilde{C})^{-1}(\widetilde{c}).& 
\end{align*}
Since $\iota = B_0B_1B_2C = (B_0B_1B_2C)^{-1}$ leaves the curves $\beta_0,\beta_1,\beta_2$ and $c$ invariant, the Dehn twists $\widetilde{B^\prime}_0, \widetilde{B^\prime}_1, \widetilde{B^\prime}_2$ and $\widetilde{C^\prime}$ along $\widetilde{\beta^\prime}_0, \widetilde{\beta^\prime}_1, \widetilde{\beta^\prime}_2$ and $\widetilde{c^\prime}$ are lifts of $B_0, B_1, B_2$ and $C$, respectively. 
From the above we have the following identities. 
\begin{align}
\partial = \widetilde{B^\prime}_0 \widetilde{B}_0 \widetilde{B}_1 \widetilde{B}_2 \widetilde{C} \widetilde{B}_1 \widetilde{B}_2 \widetilde{C} 
= (\widetilde{B}_1 \widetilde{B}_2 \widetilde{C})^2 \widetilde{B^\prime}_0 \widetilde{B}_0 . \label{liftA} \\
\partial = \widetilde{B}_0 \widetilde{B}_0 \widetilde{B}_1 \widetilde{B}_2 \widetilde{C} \widetilde{B^\prime}_1 \widetilde{B^\prime}_2 \widetilde{C^\prime} 
= \widetilde{B}_0^2 \widetilde{B}_1 \widetilde{B}_2 \widetilde{C} \widetilde{B^\prime}_1 \widetilde{B^\prime}_2 \widetilde{C^\prime} . \label{liftB}
\end{align}
Moreover, it is easy to seen that these two identities are lifts of $B_0^2(B_1B_2C)^2=1$ and $(B_1B_2C)^2B_0^2=1$, respectively. 
\begin{figure}[ht]
\begin{center}
\includegraphics[scale=.63]{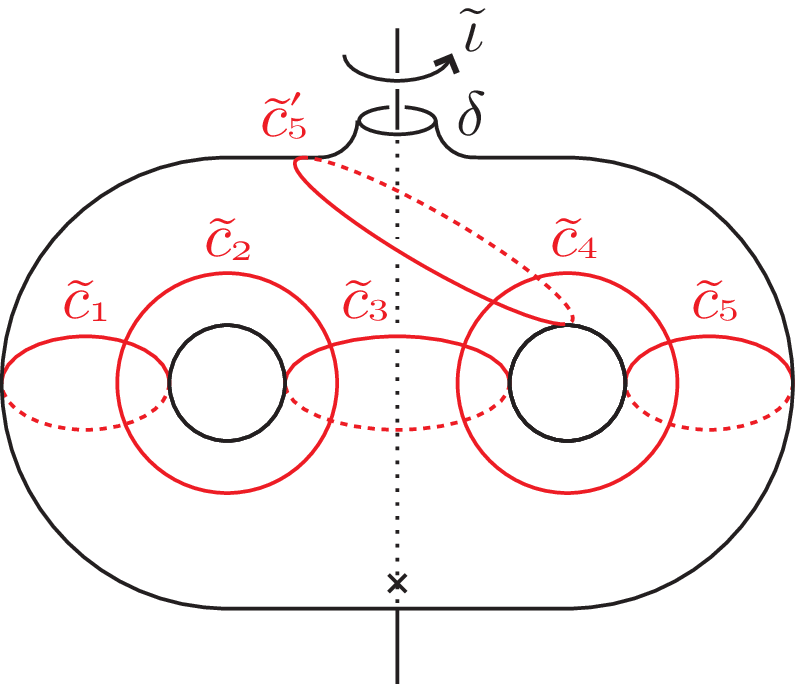}
\caption{The curves $\widetilde{c}_{1}$, $\widetilde{c}_{2}$, $\widetilde{c}_{3}$, $\widetilde{c}_{4}$, $\widetilde{c^\prime}_{4}$, $\widetilde{c}_{5}$ and $\widetilde{c^\prime}_{5}$}
\label{fig:iota3}
\end{center}
\end{figure}

Let us consider the element $\widetilde{\phi} = \widetilde{C}_4^{-1} \widetilde{C}_3^{-1} \widetilde{C}_2^{-1} \widetilde{C}_1^{-1}$ which is a lift of $\phi$ in $M_2$ to $M_2^1$, 
where $\widetilde{C}_i$ is the Dehn twist along $\widetilde{c}_i$ as shown in Figure~\ref{fig:iota3} and the element $\widetilde{\iota}=\widetilde{B}_0 \widetilde{B}_1 \widetilde{B}_2 \widetilde{C}$ which is a lift of $\iota$. 
We have
\begin{align*}
&\widetilde{v}_0 := \widetilde{\phi}(\widetilde{\beta}_0) = \widetilde{c}_5,& &\widetilde{v^\prime}_0 := \widetilde{\phi}(\widetilde{\beta^\prime}_0) = \widetilde{c}^\prime_5,& &\widetilde{w}_0 := \widetilde{\iota}\widetilde{\phi}(\widetilde{\beta}_0) = \widetilde{c}_1&
\end{align*}
and set
\begin{align*}
&\widetilde{v}_1 := \widetilde{\phi}(\widetilde{\beta}_1),& &\widetilde{v}_2 := \widetilde{\phi}(\widetilde{\beta}_2),& &\widetilde{v}_c := \widetilde{\phi}(\widetilde{c}),& \\
&\widetilde{w}_1 := \widetilde{\iota}\widetilde{\phi}(\widetilde{\beta}_1),& &\widetilde{w}_2 := \widetilde{\iota}\widetilde{\phi}(\widetilde{\beta}_2),& &\widetilde{w}_c := \widetilde{\iota}\widetilde{\phi}(\widetilde{c}),& \\
&\widetilde{w_1^\prime} := \widetilde{\iota}\widetilde{\phi}(\widetilde{\beta^\prime}_1),& &\widetilde{w_2^\prime} := \widetilde{\iota}\widetilde{\phi}(\widetilde{\beta^\prime}_2),& &\widetilde{w_c^\prime} := \widetilde{\iota}\widetilde{\phi}(\widetilde{c^\prime}).&
\end{align*}
By conjugating the relations (\ref{liftA}) and (\ref{liftB}) by $\widetilde{\phi}$ and $\widetilde{\iota}\widetilde{\phi}$, respectively, and applying the discussion above, we obtain the following two words 
\begin{align}
(\widetilde{V}_1 \widetilde{V}_2 \widetilde{V}_c)^2 \widetilde{C^\prime}_5 \widetilde{C}_5 = \partial, \label{liftC} \\
\widetilde{C}_1^2 \widetilde{W}_1 \widetilde{W}_2 \widetilde{W}_c \widetilde{W_1^\prime} \widetilde{W_2^\prime} \widetilde{W_c^\prime} = \partial. \label{liftD}
\end{align}
Therefore, we have 
\begin{align*}
(\widetilde{V}_1 \widetilde{V}_2 \widetilde{V}_c)^2 \widetilde{C^\prime}_5 \widetilde{C}_5 \widetilde{C}_1^2 \widetilde{W}_1 \widetilde{W}_2 \widetilde{W}_c \widetilde{W_1^\prime} \widetilde{W_2^\prime} \widetilde{W_c^\prime} = \partial^2. 
\end{align*}
By applying the lantern relation $\widetilde{C^\prime}_5 \widetilde{C}_5 \widetilde{C}_1^2 = \widetilde{C}_3 \widetilde{C} \widetilde{B}_2$ which is a lift of the lantern relation $C_5^2C_1^2=C_3CB_2$ in $M_2$ to $M_2^1$, we have the following equation
\begin{align*}
(\widetilde{V}_1 \widetilde{V}_2 \widetilde{V}_c)^2 \widetilde{C}_3 \widetilde{C} \widetilde{B}_2 \widetilde{W}_1 \widetilde{W}_2 \widetilde{W}_c \widetilde{W_1^\prime} \widetilde{W_2^\prime} \widetilde{W_c^\prime} = \partial^2. 
\end{align*}
It is easy to check that this is a lift of the relation (\ref{eq2:23}), that is, the Lefschetz fibration in Theorem~\ref{thm} has a sphere section of square $-2$.

Next, we show that the Lefschetz fibration in Theorem~\ref{thm1} admits a sphere section of square $-2$. 
Let us consider the element $\widetilde{\psi} = \widetilde{C}_5 \widetilde{C}_3^{-1} \widetilde{C}_2^{-1} \widetilde{C}_1^{-1}$ and the Dehn twist $\widetilde{C_4^\prime}$ along $\widetilde{c_4^\prime}$ as shown in Figure~\ref{fig:iota3} which are lifts of $\psi$ and $C_4$ in $M_2$ to $M_2^1$. 
Then, we have 
\begin{align*}
&\widetilde{x}_0 := \widetilde{\psi}(\widetilde{\beta}_0) = \widetilde{c}_4,& &\widetilde{x^\prime}_0 := \widetilde{\psi}(\widetilde{\beta^\prime}_0) = \widetilde{c_4}^\prime,&
\end{align*}
and set 
\begin{align*}
&\widetilde{x}_1 := \widetilde{\psi}(\widetilde{\beta}_1),& &\widetilde{x}_2 := \widetilde{\psi}(\widetilde{\beta}_2),& &\widetilde{x}_c := \widetilde{\psi}(\widetilde{c}),& \\
&\widetilde{a} := \widetilde{C}_5^{-1} \widetilde{C}_4^{-1}(\widetilde{\beta}_2)& &\widetilde{b} := \widetilde{C}_5^{-1} \widetilde{C}_4^{-1}(\widetilde{\beta}_2).& &&
\end{align*}
By conjugating the relations (\ref{liftA}) by $\widetilde{\psi}$ and applying the discussion above, we obtain the following word. 
\begin{align}
(\widetilde{X}_1 \widetilde{X}_2 \widetilde{X}_c)^2 \widetilde{C^\prime}_4 \widetilde{C}_4 = \partial. \label{liftE} 
\end{align}
Therefore, we have 
\begin{align}
(\widetilde{X}_1 \widetilde{X}_2 \widetilde{X}_c)^2 \widetilde{C^\prime}_4 \widetilde{C}_4 \widetilde{C}_1^2 \widetilde{W}_1 \widetilde{W}_2 \widetilde{W}_c \widetilde{W_1^\prime} \widetilde{W_2^\prime} \widetilde{W_c^\prime} = \partial^2. \label{liftF}
\end{align}
Here, since it is easily seen that 
\begin{align*}
&\widetilde{C}_5^{-1} \widetilde{C}_4^{-1}(\widetilde{c}_5) = \widetilde{c}_4,& &\widetilde{C}_5^{-1} \widetilde{C}_4^{-1}(\widetilde{c^\prime}_5) = \widetilde{c^\prime}_4,& 
\end{align*}
by conjugating the lantern relation $\widetilde{C^\prime}_4 \widetilde{C}_4 \widetilde{C}_1^2 = \widetilde{C}_3 \widetilde{C} \widetilde{B}_2$ by $\widetilde{C}_5^{-1} \widetilde{C}_4^{-1}$, we obtain the lantern relation 
\begin{align*}
\widetilde{C^\prime}_4 \widetilde{C}_4 \widetilde{C}_1^2 = \widetilde{B} \widetilde{C} \widetilde{A}, 
\end{align*}
which is a lift of the lantern relation $C_4^2 C_1^2 = BCA$ in $M_2$ to $M_2^1$. 
By applying the lantern relation $\widetilde{C^\prime}_4 \widetilde{C}_4 \widetilde{C}_1^2 = \widetilde{B} \widetilde{C} \widetilde{A}$ to the relation (\ref{liftF}), we have the following equation 
\begin{align*}
(\widetilde{X}_1 \widetilde{X}_2 \widetilde{X}_c)^2 \widetilde{B} \widetilde{C} \widetilde{A} \widetilde{W}_1 \widetilde{W}_2 \widetilde{W}_c \widetilde{W_1^\prime} \widetilde{W_2^\prime} \widetilde{W_c^\prime} = \partial^2. 
\end{align*}
It is easy to check that this is a lift of the relation (\ref{eq2:3}), that is, the Lefschetz fibration in Theorem~\ref{thm1} has a sphere section of square $-2$. 

Finally, we show that the Lefschetz fibration in Theorem~\ref{thm2} admits a sphere section of square $-2$. 
Let us consider the element $\widetilde{\phi}_n = \widetilde{C}_1^n \widetilde{C}_4$ and $\widetilde{\psi}_n = \widetilde{C}_1^{n} \widetilde{C}_5^{-1} \widetilde{C}_4^{-1} \widetilde{C}_3^{-1} \widetilde{C}_2^{-1} \widetilde{C}_1^{-1}$ which are lifts of $\phi_n$ and $\psi_n$ in $M_2$ to $M_2^1$, respectively. 
Then, we have 
\begin{align*}
&\widetilde{y}_{c,n} := \widetilde{\phi}_n(\widetilde{c}) = \widetilde{c},& &\widetilde{y}_{c,n} := \widetilde{\phi}_n(\widetilde{c^\prime}) = \widetilde{c^\prime},& \\
\end{align*}
and set 
\begin{align*}
&\widetilde{y}_{0,n} := \widetilde{\phi}_n(\widetilde{\beta}_0),& &\widetilde{y}_{1,n} := \widetilde{\phi}_n(\widetilde{\beta}_1),& &\widetilde{y}_{2,n} := \widetilde{\psi}_n(\widetilde{\beta_2}),& \\
&& &\widetilde{y^\prime}_{1,n} := \widetilde{\phi}_n(\widetilde{\beta^\prime}_1),& &\widetilde{y^\prime}_{2,n} := \widetilde{\psi}_n(\widetilde{\beta^\prime}_2),& \\
&\widetilde{z}_{3,n} := \widetilde{\psi}_n(\widetilde{c}_3),& &\widetilde{z}_{2,n} := \widetilde{\psi}_n(\widetilde{\beta}_2),& &\widetilde{z}_{c,n} := \widetilde{\psi}_n(\widetilde{c}).& 
\end{align*}
By conjugating the relations (\ref{liftA}) by $\widetilde{\phi}_n$ and applying the discussion above, we obtain the following word. 
\begin{align*}
(\widetilde{Y}_{0,n})^2 \widetilde{Y}_{1,n} \widetilde{Y}_{2,n} \widetilde{C} \widetilde{Y^\prime}_{1,n} \widetilde{Y^\prime}_{2,n} \widetilde{C^\prime}  = \partial. 
\end{align*}
Therefore, we have 
\begin{align}
(\widetilde{X}_1 \widetilde{X}_2 \widetilde{X}_c)^2 \widetilde{C^\prime}_4 \widetilde{C}_4 (\widetilde{Y}_{0,n})^2 \widetilde{Y}_{1,n} \widetilde{Y}_{2,n} \widetilde{C} \widetilde{Y^\prime}_{1,n} \widetilde{Y^\prime}_{2,n} \widetilde{C^\prime} = \partial^2. \label{liftH}
\end{align}
Here, since it is easily seen that 
\begin{align*}
&\widetilde{\psi}_n (\widetilde{c}_5) = \widetilde{c}_4,& &\widetilde{\psi}_n (\widetilde{c^\prime}_5) = \widetilde{c^\prime}_4,& &\widetilde{\psi}_n (\widetilde{c}_1) = \widetilde{y}_{0,n},&
\end{align*}
by conjugating the lantern relation $\widetilde{C^\prime}_5 \widetilde{C}_5 \widetilde{C}_1^2 = \widetilde{C}_3 \widetilde{C} \widetilde{B}_2$ by $\widetilde{\psi}_n$, we obtain the lantern relation 
\begin{align*}
\widetilde{C^\prime}_4 \widetilde{C}_4 \widetilde{Y}_{0,n}^2 = \widetilde{Z}_{3,n} \widetilde{Z}_{c,n} \widetilde{Z}_{2,n}, 
\end{align*}
which is a lift of the lantern relation $C_4^2 Y_{0,n}^2 = Z_{3,n}Z_{c,n}Z_{2,n}$ in $M_2$ to $M_2^1$. 
By applying the lantern relation $\widetilde{C^\prime}_4 \widetilde{C}_4 \widetilde{Y}_{0,n}^2 = \widetilde{Z}_{3,n} \widetilde{Z}_{c,n} \widetilde{Z}_{2,n}$ to the relation (\ref{liftH}), we have the following equation 
\begin{align*}
(\widetilde{X}_1 \widetilde{X}_2 \widetilde{X}_c)^2 \widetilde{Z}_{3,n} \widetilde{Z}_{c,n} \widetilde{Z}_{2,n} \widetilde{Y}_{1,n} \widetilde{Y}_{2,n} \widetilde{C} \widetilde{Y^\prime}_{1,n} \widetilde{Y^\prime}_{2,n} \widetilde{C^\prime} = \partial^2. 
\end{align*}
It is easy to check that this is a lift of the relation (\ref{eq2:3}), that is, the Lefschetz fibration in Theorem~\ref{thm2} has a sphere section of square $-2$. 

\section{Fiber sums of Matsumoto's and Xiao's genus two Lefschetz fibrations}\label{twisted}

In this section, using Matsumoto's and Xiao's \cite{Ma, Xiao} genus two Lefschetz fibrations, we construct various examples of genus two Lefschetz fibrations over $\mathbb{S}^2$ via the symplectic fiber sums and Luttinger surgery with the invariants $c_1^{2} = 1, 2$,  $\chi = 1$ and the fundamental groups $1$, $\mathbb{Z} \times \mathbb{Z}$,  $\mathbb{Z}_{n}$,  and $\mathbb{Z} \times \mathbb{Z}_{n}$ (for any integer $n \geq 2$). In the simply connected cases of $c_1^{2} = 1$ and $\chi = 1$, and  $c_1^{2} = 2$ and $\chi = 1$, the total spaces of our Lefschetz fibrations are exotic copies of $\CP\#8\CPb$ and $\CP\#7\CPb$, respectively. As we remarked in the introduction, our construction method and building blocks are similar to the ones given \cite{A, AP1, ABP, AP2, AO, AZ}. These Lefschetz fibrations were promised in the previous version of this article (see Remark~14, pages 14-15, arXiv:GT/1509.01853v2). Similar examples were obtained recently and independently by I. Baykur and M. Korkmaz in \cite{BK}, but also using some of the symplectic building blocks and constructions given in \cite{AP, ABP, AP2}. 

\subsection{Small genus two Lefschetz fibrations with $c_1^{2} = 1$, $\chi = 1$}\label{chi=1}

In this section, we construct various genus two Lefschetz fibrations with the topological invariants $e = 11$ and $\sigma = -7$. Let us recall from subsection~\ref{m} that there is a genus $2$ symplectic surface $\widetilde{\Sigma}_2$ of self-intersection $0$ in $(\mathbb{T}^2\times \mathbb{S}^2)\#4\CPb$, where $\widetilde{\Sigma}_2$ is a regular fiber of Matsumoto's fibration. We will denote the standard generators of $\pi_1(\widetilde{\Sigma}_2)$ and $\pi_1((\mathbb{T}^2\times \mathbb{S}^2)\#4\CPb)\cong \pi_1(\mathbb{T}^2)$ by $\widetilde{a}_i,\widetilde{b}_i$ ($i=1,2$) and $x,y,$ respectively. Using the expressions for the vanishing cycles $\beta_0$ and $\beta_2$ in Section~\ref{m}, we can assume that the inclusion $\widetilde{\Sigma}_2 \hookrightarrow (T^2\times S^2)\#4\CPb$ maps the fundamental group generators as follows:  
\begin{equation}\label{eq: embedding of tilde{Sigma}_2}
\widetilde{a}_1 \mapsto x, \ \ 
\widetilde{b}_1 \mapsto y, \ \
\widetilde{a}_2 \mapsto x^{-1}, \ \ 
\widetilde{b}_2 \mapsto y^{-1}.  
\end{equation}
Moreover, it is easy to see that $\pi_1(((\mathbb{T}^2\times \mathbb{S}^2)\#4\CPb)\setminus\nu\widetilde{\Sigma}_2)\cong \pi_1((\mathbb{T}^2\times \mathbb{S}^2)\#4\CPb)\cong \mathbb{Z}^2$. This follows from the fact that Matsumoto's fibration admits a sphere section and hence the meridian of $\widetilde{\Sigma}_2$ is nullhomotopic in the complement of $\widetilde{\Sigma}_2$. 

Let us now recall from subsection~\ref{x} that there is a genus $2$ symplectic surface $\bar{\Sigma}_2$ of self-intersection $0$ in $S(E,3) = (\mathbb{T}^2\times \mathbb{S}^2)\#3\CPb$, where $\bar{\Sigma}_2$ is a regular fiber of Xiao's genus two fibration. We denote the standard generators of $\pi_1(\bar{\Sigma}_2)$ and $\pi_1((\mathbb{T}^2\times \mathbb{S}^2)\#3\CPb)\cong \pi_1(\mathbb{T}^2)\cong \mathbb{Z}^2$ by $\bar{c}_i,\bar{d}_i$ ($i=1,2$) and $z,t,$ respectively. Without loss of generality, we can assume that one of the nonseparating vanishing cycle of Xiao's fibration can be reperesented by the loop $\gamma = \bar{d_{1}}\bar{d_{2}}$ in $\pi_1(\bar{\Sigma}_2)$. This can be achieved by conjugating the global monodromy of Xiao's fibration by an isotopy class of a carefully chosen diffeomorphism of $\bar{\Sigma_2}$. Thus, we can assume that the following relation holds in $\pi_1((\mathbb{T}^2\times \mathbb{S}^2)\#3\CPb)$: $\bar{d_{1}}\bar{d_{2}}=1$. Here, we can replace the relation $\bar{d_{1}}\bar{d_{2}}=1$ with $\bar{c_{1}}\bar{c_{2}}=1$ if needed. Moreover, since Xiao's fibration admits a sphere section~\ref{x}, it's fundamental group is carried by a regular fiber, and $\pi_1(((\mathbb{T}^2\times \mathbb{S}^2)\#3\CPb)\setminus\nu\bar{\Sigma}_2)\cong \pi_1((\mathbb{T}^2\times \mathbb{S}^2)\#3\CPb)\cong \mathbb{Z}^2$. Since $\pi_1((\mathbb{T}^2\times \mathbb{S}^2)\#3\CPb)$ is a free abelian group of rank two and  $\bar{c_1}$, $\bar{d_1}$, $\bar{c_2}$, $\bar{d_2}$ generate this group, we can assume that it can be generated by $\{\bar{d_1}$, $\bar{c_1}\}$ or $\{\bar{d_1}, \bar{c_2}\}$.  This follows from the basic fact that in a free abelian group any two basis have the same cardinality.

We form the symplectic fiber sums of Matsumoto's fibration and Xiao's fibration along their regular fibers: $Z(\phi)=(\mathbb{T}^2\times \mathbb{S}^2\#4\CPb)\#_{\phi} (\mathbb{T}^2\times \mathbb{S}^2\#3\CPb)$ and $Z(\psi)=(\mathbb{T}^2\times \mathbb{S}^2\#4\CPb)\#_{\psi} (\mathbb{T}^2\times \mathbb{S}^2\#3\CPb)$, where the gluing diffeomorphisms $\psi, \phi: \partial(\nu \widetilde{\Sigma}_2) \rightarrow \partial(\nu \bar{\Sigma}_2)$ maps
the generators of $\pi_1(\widetilde{\Sigma}_2^{\scparallel})$ as follows:  

\begin{equation}\label{eq: embedding of tilde{Sigma}_2}
\widetilde{a}_1^{\scparallel} \mapsto \bar{c}_1^{\scparallel}, \ \ 
\widetilde{b}_1^{\scparallel} \mapsto \bar{d}_1^{\scparallel}, \ \
\widetilde{a}_2^{\scparallel} \mapsto \bar{c}_2^{\scparallel}, \ \ 
\widetilde{b}_2^{\scparallel} \mapsto \bar{d}_2^{\scparallel},
\end{equation}

\begin{equation}\label{eq: embedding of tilde{Sigma}_2}
\widetilde{a}_1^{\scparallel} \widetilde{b}_1^{\scparallel} \mapsto \bar{c}_1^{\scparallel}, \ \ 
\widetilde{a}_1^{\scparallel-1} \mapsto \bar{d}_1^{\scparallel}, \ \
\widetilde{a}_2^{\scparallel} \mapsto \bar{c}_2^{\scparallel}, \ \ 
\widetilde{b}_2^{\scparallel} \mapsto \bar{d}_2^{\scparallel},
\end{equation}

It follows from Seifert-Van Kampen's Theorem that $\pi_1(Z(\phi))$ and $\pi_1(Z(\psi))$ are both abelian groups of rank less than equal $2$. Since the gluing map $\phi$ introduces no new relations, among the generators $c_1$ and $d_1$, the fundamental group of $Z(\phi)$ is $\mathbb{Z} \times \mathbb{Z}$. To compute the fundamental group of $Z(\psi)$, we use the relations $1 = d_1d_2 = a_1b_2$, $a_1a_2 = b_1b_2=1$, we deduce that the fundamental group of is $\mathbb{Z}$. The Euler characteristic and signature of  $Z(\phi)$ and $Z(\psi)$ are given as follows: $e(Z(\phi); Z(\psi)) = e(\mathbb{T}^2\times \mathbb{S}^2\#4\CPb) + e(\mathbb{T}^2\times \mathbb{S}^2\#3\CPb) + 4 = 11$, and $\sigma(Z(\phi); Z(\psi)) = \sigma(\mathbb{T}^2\times \mathbb{S}^2\#4\CPb) +\sigma (\mathbb{T}^2\times \mathbb{S}^2\#3\CPb)  = -7$, and thus $c_1^2(Z(\phi); Z(\psi) = \chi(Z(\phi); Z(\psi)) = 1$. 

Let $\mu$ and $\mu'$ denote two distinct meridians of $\widetilde{\Sigma}_2$, which are nullhomotpic in the complement $(\mathbb{T}^2\times \mathbb{S}^2\#4\CPb) \setminus \nu \bar{\Sigma}_2$. Let $R = \tilde{d_2} \times \mu$ be the ‘rim tori’ in $Z(\psi)$ resulting from the symplectic fiber sum, where $\tilde{d_2}$ is a suitable parallel copy of the generators $d_2$. Observe that this rim torus is Lagrangian, and has a dual rim torus $T = \tilde{c_2} \times \mu'$. We $[R]^2 = [T]^2 = 0$, and $[R] \cdot [T]$ = $1$. One can show that the effect of a Luttinger surgery $(R,  \tilde{d_2}, \pm 1)$ is the same as changing the gluing map that we have used in our symplectic fiber sum above in the construction of $Z(\psi)$. Let us perform the following Luttinger surgery on Lagrangian torus $(R,  \tilde{d_2}, -1/n)$. The Luttinger surgery yields to $d_2^n = 1$. We obtain the following abelian groups of rank less than equal one as the fundamental groups: (a) $\{0\}$ if we set $n = \pm 1$; (b) $\mathbb{Z}_{n}$ if set $n= \pm |n|$ where $|n| \geq 2$. The Lefschetz fibration with  $c_1^{2} = 2$ and $\chi = 1$ and with the fundamental group $\mathbb{Z} \times \mathbb{Z}_{n}$ can be otained by applying a single Luttinger surgery to $Z(\phi)$. We skip the details since they are very similar to the above constructions using $Z(\psi)$.

\subsection{Small Lefschetz fibrations with $c_1^{2} = 2$ and $\chi = 1$}

In this section, we will construct the genus two Lefschetz fibrations with the invariants $c_1^{2} = 2$, $\chi = 1$, and the fundamental groups $1$, $\mathbb{Z} \times \mathbb{Z}$,  $\mathbb{Z}_{n}$, and $\mathbb{Z} \times \mathbb{Z}_{n}$ (for any $n \geq 2$). When the fundamental group is trivial, the total space of our Lefschetz fibration is an exotic copy of $\CP\#7\CPb$. Using the symplectic fiber sum, we will first construct our Lefschetz fibration with the invariants $c_1^{2} = 2$, $\chi = 1$, and the fundamental group $\mathbb{Z} \times \mathbb{Z}$. By applying one or two Luttinger surgeries to the total space of this fibration, we will obtain our Lefschetz fibrations with $c_1^{2} = 2$, $\chi = 1$, and the fundamental groups $1$, $\mathbb{Z}_{n}$,  and $\mathbb{Z} \times \mathbb{Z}_{n}$ (for any $n \geq 2$). We refer the reader to \cite{ABP, AZ} for similar constructions of Lefschetz fibrations with $c_1^{2} = 0$, and $\chi = 1$, which were obtained via Luttinger surgery on fiber sum of two copies of Matsumoto's fibration. The expert reader will observe that the later genus two fibration has the total space $E(1)\#(\mathbb{T}^2\times \mathbb{T}^2)$, and the genus two fibration obtained in \cite{ABP, AZ} via Luttinger surgery decends from this genus two fibration (see also related earlier work in~\cite{FS2}).

We form the symplectic fiber sum of two copies of Xiao's fibration along their regular fibers: $W(\psi)=(\mathbb{T}^2\times \mathbb{S}^2\#3\CPb)\#_{\psi} (\mathbb{T}^2\times \mathbb{S}^2\#3\CPb)$, where the gluing diffeomorphism $\psi: \partial(\nu \widetilde{\Sigma}_2) \rightarrow \partial(\nu \bar{\Sigma}_2)$ maps the generators of $\pi_1(\widetilde{\Sigma}_2^{\scparallel})$ as follows:  

\begin{equation}\label{eq: embedding of tilde{Sigma}_2}
\bar{c}_1^{\scparallel} \mapsto \bar{c}_1'^{\scparallel}, \ \ 
\bar{d}_1^{\scparallel} \mapsto \bar{d}_1'^{\scparallel}, \ \
\bar{c}_2^{\scparallel} \mapsto \bar{c}_2'^{\scparallel}, \ \ 
\bar{d}_2^{\scparallel} \mapsto \bar{d}_2'^{\scparallel},
\end{equation}

Since Xiao's fibration admits $-1$ sphere sections, the genus two fibration on $W(\psi)$ admits $-2$ sphere section. Also, recall from our discussion in Section , we can assume that the relation $d_{1}d_{2}=1$ holds in $\pi_1((\mathbb{T}^2\times \mathbb{S}^2)\#3\CPb)$ for the first fibration, and the relation $c_{1}'c_{2}'=1$ holds in $\pi_1((\mathbb{T}^2\times \mathbb{S}^2)\#3\CPb)$ for the second fibration. Using the existence of $-2$ sphere section, it follows from Seifert-Van Kampen's Theorem that $\pi_1(W(\psi))$ is generated by $c_{1}$, $d_{1}$, $c_{2}$, and $d_{2}$, and the relations $c_{1}c_{2}=1$ and $d_{1}d_{2}=1$ both hold in $\pi_1(Z(\psi))$. Moreover, using the relations
$d_1d_2= c_1c_2=1$, we deduce that the fundamental group of $W(\psi)$ is $\mathbb{Z} \times \mathbb{Z}$. The Euler characteristic and signature of $W(\psi)$ are given as follows: $e(W(\psi)) = e(\mathbb{T}^2\times \mathbb{S}^2\#3\CPb) + e(\mathbb{T}^2\times \mathbb{S}^2\#3\CPb) + 4 = 10$, and $\sigma(W(\psi)) = \sigma(\mathbb{T}^2\times \mathbb{S}^2\#3\CPb) +\sigma (\mathbb{T}^2\times \mathbb{S}^2\#3\CPb)  = -6$, and thus $c_1^2(W(\psi) =2$ and $\chi(W(\psi) = 1$.

Let $\mu$ and $\mu'$ denote two distinct meridians of $\bar{\Sigma}_2$, which are nullhomotpic in the complement $(\mathbb{T}^2\times \mathbb{S}^2\#3\CPb) \setminus \nu \bar{\Sigma}_2$. Let $R_{1} = \tilde{c_1} \times \mu$, and $R_{2} = \tilde{d_2} \times \mu$ be the ‘rim tori’ in $W(\psi)$ resulting from the symplectic fiber sum, where $\tilde{c_1}$ and $\tilde{d_2}$ are suitable parallel copies of the generators $c_1$ and $d_2$. Observe that the rim tori $R_{1}$ and $R_{2}$ are Lagrangian. Moreover, these rim tori have the dual rim tori $T_{1} = \tilde{d_1} \times \mu'$, and $T_{2} = \tilde{c_2} \times \mu'$ in $W(\psi)$. Note that $[R_{1}]^2 = [R_{2}]^2 = [T_1]^2 = [T_2]^2 = 0$, and $[R_1] \cdot [T_1]$ = $1$ = $[R_2] \cdot [T_2]$. One can show that the effect of two Luttinger surgeries $(R_1,  \tilde{c_1}, \pm 1/m)$ and $(R_2, \tilde{d_2}, \pm 1/n)$ is the same as changing the gluing map that we have used in our symplectic fiber sum above in the construction of $W(\psi)$. Let us perform the following two Luttinger surgeries on pairwise disjoint Lagrangian tori $(R_1,  \tilde{c_1}, -1)$ and $(R_2, \tilde{d_2},-1/n)$. Our first Luttinger surgery yields to $c_1^m = 1$ and the second surgery yields to the relation  $d_2^n = 1$. Since $c_1$ and $d_2$ also commutes, we obtain the following abelian groups of rank less than equal one as the fundamental groups: (a) $\{0\}$ if we set $n = \pm 1$, $m = \pm 1$, (b) $\mathbb{Z}_{n}$ if set $m = \pm 1$, and $n= \pm |n|$ where $|n| \geq 2$, and (c) $\mathbb{Z} \times \mathbb{Z}_{n}$ if we set $m=0$ and $n= \pm |n|$ where $|n| \geq 2$. 

\subsection{Small Lefschetz fibrations with $b_2^+ = 3$}

In this section, present the constructions of genus two Lefschetz fibrations over $\mathbb{S}^2$ with $b_2^+ = 3$, $b_2^- = 12, 13, 14, 15$ and with a trivial fundamental group, using the Lefschetz fibration $f_{-1}: X_{\theta_{-1}} \rightarrow \mathbb{S}^2$ obtained in Section~\ref{construction} along with Matsumoto's and Xiao's genus two Lefschetz fibrations. The total spaces of our genus two Lefschetz fibrations are exotic copies of $3\CP\#12\CPb$,  $3\CP\#13\CPb$,  $3\CP\#14\CPb$, and $3\CP\#15\CPb$, respectively. Recall that existence of such exotic symplectic $4$-manifolds are well known (see introduction of \cite{AP1}, and \cite{AP2} for concise history). Some of these exotic $4$-manifolds, but not all, were not known to be a total space of Lefschetz fibration over $\mathbb{S}^2$. Here we give yet another construction, which demonstrates them as the total spaces of genus two Lefschetz fibrations over $\mathbb{S}^2$.

We first present the cases of $b_2^+ = 3$ and $b_2^- = 13, 14$, which are obtained similarly. The total spaces of these Lefschetz fibrations are exotic copies of $3\CP\#13\CPb$ and $3\CP\#14\CPb$. To obtain such Lefschetz fibrations, we form the symplectic fiber sum of the fibration Lefschetz $f_{-1}: X_{\theta_{-1}} \rightarrow \mathbb{S}^2$ with Matsumoto's fibration and Xiao's fibration respectively, along their regular fibers: $V_{14} :=X_{\theta_{-1}}\#_{\Sigma_2} (\mathbb{T}^2\times \mathbb{S}^2\#4\CPb)$ and $V_{13} :=X_{\theta_{-1}}\#_{\Sigma_2} (\mathbb{T}^2\times \mathbb{S}^2\#3\CPb)$, 


Using the facts that all three Lefschetz fibrations admit sphere sections and applying the Seifert - Van Kampen's Theorem, it is straightforward to check the fundamental groups of $V_{14}$ and $V_{13}$ are trivial. Since $V_{14}$ and $V_{13}$ are symplectic 4-manifolds with $b_2^+ = 3$, and $e(V_{14}) = 19$, $\sigma(V_{14}) = -11$, and $e(V_{13}) = 18$, $\sigma(V_{13}) = -10$, using the Freedman's classification theorem~\cite{freedman} for simply-connected 4-manifolds and nontriviality of Seiberg-Witten invariants for symplectic 4-manifolds~\cite{taubes}, we conclude that $V_{13}$ and $V_{14}$ are exotic copies of $3\CP\#13\CPb$, and $3\CP\#14\CPb$ respectively. It follows from Usher's Minimality Theorem in ~\cite{U} that $V_{13}$ and $V_{14}$ are minimal symplectic 4-manifolds.

The genus two Lefschetz fibration with $b_2^+ = 3$, $b_2^- = 15$, and with an exotic minimal total space can be obtained in several ways. For example, one can form the symplectic fiber sum of genus two Lefschetz fibration on the knot surgered elliptic surface $E(1)_K$ \cite{FS2}, where $K$ is the trefoil or figure eight knot, with Matsumoto's genus two fibration on $\mathbb{T}^2\times \mathbb{S}^2\#4\CPb$. Another example was constructed in \cite{AP} applying the lantern relations to the monodromy of a well known genus two Lefschetz fibration on $K3\#2\CPb$. Here we present the third example by taking $3$-fold symplectic fiber sum $V_{15} := (\mathbb{T}^2\times \mathbb{S}^2\#4\CPb)\#_{\phi} (\mathbb{T}^2\times \mathbb{S}^2\#4\CPb)\#_{\psi} (\mathbb{T}^2\times \mathbb{S}^2\#4\CPb) = M\#_{\phi}M\#_{\psi}M$ of Matsumoto's genus two fibration, where our gluing maps $\phi$ and $\psi$ send the generators $a_1$ and $b_1$ of  $\pi_{1}(M) = \ \langle a_{1}, b_{1}, a_{2}, b_{2} \ | \  b_{1}b_{2} = [a_{1}, b_{1}] = [a_{2}, b_{2}] = b_{2}a_{2}{b_{2}}^{-1}a_{1} = 1 \rangle $ to the nonseparating vanishing cycles of Matsumoto's fibration. More precisely, we choose our gluing diffeomorphism $\phi$ such that it maps the nonseparating loop $a_1$ and the dual nonseparating vanishing cycle $\beta_0$ in the first copy to the nonseparating vanishing cycle $\beta_0'$ and the dual nonseparating loop $a_1'$ in the second copy. The existence of such a diffeomorphism is easy to show, and can be found in \cite{FM} (see pages 39-40). Similarly, we choose our gluing map $\psi$ such that it maps the nonseparating loop $b_1'$ in $M\#_{\phi}M$, which survives and generates $\pi_{1}(M\#_{\phi}M)$, to the nonseparating vanishing cycle $\beta_0''$ in the third copy. It is easy to see that $\pi_{1}(V_{15})$ is trivial. Applying Freedman's classification theorem~\cite{freedman} and the fact that $V_{15}$ is symplectic 4-manifold with $b_2^+ = 3$, we conclude that $V_{15}$ is an exotic copy of of $3\CP\#15\CPb$. By Usher's Minimality Theorem in ~\cite{U}, we conclude that that $V_{15}$ is minimal symplectic 4-manifold. 

Our genus two Lefschetz fibration with $b_2^+ = 3$, $b_2^- = 12$, and with an exotic total  space will be obtained by taking $3$-fold symplectic fiber sum $V_{12}$ := $(\mathbb{T}^2\times \mathbb{S}^2\#3\CPb)\#_{\phi} (\mathbb{T}^2\times \mathbb{S}^2\#3\CPb)\#_{\psi} (\mathbb{T}^2\times \mathbb{S}^2\#4\CPb)$ = 
$S\#_{\phi}S\#_{\psi}S$ of Xiao's genus two Lefschetz fibration. Our gluing maps $\phi$ and $\psi$ are identical to that given above. Using the discussion in Section~\ref{chi=1}, we see that the gluing maps $\phi$ and $\psi$ are well defined; one can assume that $\beta_0$ is one of the vanishing cycle of Xiao's fibration by conjugating it's monodromy with the mapping class of a carefully chosen diffeomorphism of $\bar{\Sigma_2}$. Our fundamental group computation for $S\#_{\phi}S\#_{\psi}S$ follows the exact same line of arguments as in the case for $M\#_{\phi}M\#_{\psi}M$. Using Freedman's classification theorem~\cite{freedman}, the fact that $V_{15}$ is symplectic 4-manifold with $b_2^+ = 3$, and Usher's Theorem, we can conclude that $V_{12}$ is an exotic copy of of $3\CP\#12\CPb$, and symplectically minimal.

\subsection{Construction of exotic $\CP\#4\CPb$ and $3\CP\#6\CPb$ using Xiao's genus two fibration}\label{sec:exotic 1-4}

Our construction of exotic $\CP\#4\CPb$ and $3\CP\#6\CPb$ presented here are analogs of the constructions of exotic copies $\CP\#5\CPb$ and $3\CP\#7\CPb$ obtained by the first author in \cite{A}. In \cite{A}, the first author used Matsumoto's genus two Lefschetz fibration on $\mathbb{T}^2\times \mathbb{S}^2\#4\CPb$ along with other symplectic building blocks that he constructed in \cite{A1} to produce the above mentioned small exotic $4$-manifolds (see also follow up work in \cite{AP1, ABP, AP2} where further progress on this problem was made and alternative constructions were presented).

In this section, using Xiao's genus two fibration on $\mathbb{T}^2\times \mathbb{S}^2\#3\CPb$ given in Section~\ref{x}, and the symplectic building blocks in \cite{AP2} (see pages 2-3), we construct an irreducible smooth $4$-manifolds homeomorphic but not diffeomorphic to $\CP\#4\CPb$ and $3\CP\#6\CPb$. Let us recall that exotic irreducible smooth structures on $\CP\#4\CPb$ and $3\CP\#6\CPb$ were already constructed in \cite{AP2} (the reader is refered to the Sections 10 and 12). The family of exotic smooth structures constructed in \cite{AP2}, contains one symplectic and infinitely many non-symplectic exotic copies of $\CP\#4\CPb$ and $3\CP\#6\CPb$. Our construction here is similar, but we have one new ingredient, which is the use of Xiao's fibration. It is not clear if the exotic irreducible symplectic (smooth) 4-manifolds obtained here are diffeomorphic to the exotic irreducible symplectic (smooth) $4$-manifolds constructed in \cite{AP2}; the building blocks and symplectic submanifolds are different. To construct our exotic smooth $\CP\#4\CPb$, we use the smooth $4$-manifolds $Y_{2}(1,m)$ obtained from $\Sigma_2\times \mathbb{T}^2$ by applying four torus surgeries, three of which are Luttinger surgeries (see \cite{AP2}). Let us recall the construction of $Y_{2}(1,m)$ from \cite{AP2}. We fix an integer $m \geq 1$. Assume that $\Sigma_2 \times \mathbb{T}^{2}$ is equipped with a product symplectic form. Let $Y_{2}(1,m)$ denote smooth $4$-manifold obtained by performing the following $4$ torus surgeries on $\Sigma_2\times \mathbb{T}^2$:

\begin{eqnarray}\label{eq: Luttinger surgeries for Y_1(m)} 
&&(a_1' \times c', a_1', -1), \ \ (b_1' \times c'', b_1', -1),\\  \nonumber
&&(a_{2}' \times c', c', +1), \ \ (a_{2}'' \times d', d', +m).
\end{eqnarray}

\noindent where $a_i,b_i$ ($i=1,2$) and $c,d$\/ denote the standard generators of $\pi_1(\Sigma_{2})$ and $\pi_1(\mathbb{T}^2)$, respectively. Since all the torus surgeries listed above are Luttinger surgeries when $m =1$ and the Luttinger surgery preserves minimality, $Y_{2}(1,1)$ is a minimal symplectic $4$-manifold. The fundamental group of $Y_{2}(1,m)$ is generated by $a_i,b_i$ ($i=1,2$) and $c,d$, and the following relations hold in $\pi_1(Y_{2}(1,m))$:

\begin{gather}\label{Luttinger relations for Y_1(m)}
[b_1^{-1},d^{-1}]=a_1,\ \  [a_1^{-1},d]=b_1,\ \
[d^{-1},b_{2}^{-1}]=c,\ \ {[c^{-1},b_{2}]}^{-m}=d,\\ \nonumber
[a_1,c]=1,\ \  [b_1,c]=1,\ \ [a_2,c]=1,\ \  [a_2,d]=1,\\ \nonumber
[a_1,b_1][a_2,b_2]=1,\ \ [c,d]=1.
\end{gather}

Let us denote by $\Sigma'_2 \subset Y_{2}(1,m)$ a genus $2$ surface that desend from the surface $\Sigma_{2}\times\{{\rm pt}\}$ in $\Sigma_{2}\times \mathbb{T}^2$. To construct exotic $\CP\#4\CPb$, we start with this genus 2 surface of self-intersection zero in $\Sigma'_2$ in $Y_{2}(1,m)$. We form the normal connected sum $X(m)=Y_{2}(1,m)\#_{\psi} \mathbb{T}^2\times \mathbb{S}^2\#3\CPb$, where the gluing diffeomorphism $\psi:\partial(\nu {\Sigma'_2})\rightarrow\partial(\nu\bar{\Sigma}_2)$ maps the generators of $\pi_1({\Sigma'_2}^{\scparallel})$ and $\pi_1({\bar{\Sigma_2}}^{\scparallel})$as follows:
\begin{equation}\label{eq:psi mapsto for 1-4}
\hat{a}_1^{\scparallel} \mapsto \bar{c}_1^{\scparallel}, \ \ 
\hat{b}_1^{\scparallel} \mapsto \bar{d}_1^{\scparallel}, \ \
\hat{a}_2^{\scparallel} \mapsto \bar{c}_2^{\scparallel}, \ \ 
\hat{b}_2^{\scparallel} \mapsto \bar{d}_2^{\scparallel}, \ \
\end{equation}

Notice that $X(m)$ is symplectic when $m=1$.  

\begin{lem}\label{lem:X(m)}
The family\/ $\{ X(m)\mid m\geq 1\}$ consists of irreducible\/ smooth $4$-manifolds that all are homeomorphic to $\CP\#4\CPb$. Moreover, this family contains an infinite subfamily consisting of pairwise non-diffeomorphic non-symplectic\/ $4$-manifolds.  
\end{lem}

\begin{proof}
We compute that 
\begin{eqnarray*}
e(X(m)) &=& e(Y_{2}(1,m)) + e(\mathbb{T}^2\times \mathbb{S}^2\#3\CPb) -2 e(\Sigma_{2}) = 0 + 3 + 4 = 7,\\
\sigma(X(m)) &=& \sigma (Y_{2}(1,m)) + \sigma(\mathbb{T}^2\times \mathbb{S}^2\#3\CPb) = 0 + (-3) = -3.
\end{eqnarray*}
Freedman's theorem (cf.\ \cite{freedman}) implies that $X(m)$ is homeomorphic to $\CP\# 4\CPb$, once we verify $\pi_1(X(m))=1$. Applying Seifert-Van Kampen's theorem, we deduce that $\pi_1(X(m))$ is a quotient of the following finitely presented group:  
\begin{eqnarray}\label{eq: pi_1(X(m))}
\hspace{15pt}
\langle
a_1,b_1,a_2,b_2, c, d
&\mid& a_1=[b_1^{-1},d^{-1}],\, b_1=[a_1^{-1},d],\\ \nonumber
&\mid& c=[b_2^{-1}, d^{-1}],\, d=[c^{-1},b_2]^{-m},\\ \nonumber
&&[a_1,c]=[b_1,c]=[a_2,c]=[a_2,d] =1,\\ \nonumber
&&[a_1,b_1][a_2,b_2]=1,\, [c,d]=1, b_1b_2=1 \nonumber \rangle
\end{eqnarray}
Using the relations $b_1b_2 = 1$, $d=[c^{-1},b_2]^{-m}$,  and $[b_1,c] =1$, it is easy to see that $d=1$ in $\pi_1(X(m))$, which implies that the generators $a_1 = b_1 = c = 1$ of (\ref{eq: pi_1(X(m))}) are also trivial in $\pi_1(X(m))$.  To show $a_2 = 1$, we recall from \ref{chi=1} that $\pi_1((\mathbb{T}^2\times \mathbb{S}^2)\#3\CPb)$ is a free abelian group of rank two and $\bar{c_1}$, $\bar{d_1}$ generate this group, thus one can express the element $\bar{c_2}$ as  $\bar{c_1}^{k}\bar{d_1}^{l}$ in this group, for some integers $k$ and $l$. Using the gluing map $\psi$ above, we identify $a_2 = c_2 = \bar{c_1}^{k}\bar{d_1}^{l} = a_{1}^{k}b_{1}^{l}$. Thus $a_2 = 1$ in $\pi_1(X(m))$ as well. It follows that $\pi_1(X(m))=1$.   The same arguments as in the proof of Lemma~13 in \cite{AP2} show that $X(m)$ are all irreducible and infinitely many of them are non-symplectic and pairwise non-diffeomorphic.  
\end{proof}

\begin{lem}\label{lem:3-k}
There exist an irreducible symplectic\/ $4$-manifold and an infinite family of pairwise non-diffeomorphic irreducible non-symplectic\/ $4$-manifolds that are all homeomorphic to\/ $3\CP\#6\CPb$.
\end{lem}

\begin{proof} Recall that such an infinite family of exotic $3\CP\#6\CPb$'s was already constructed by the first author and D. Park in \cite{AP2} (see Section~12). The existence of exotic $3\CP\#6\CPb$ immediately follows from the existence of exotic $\CP\#4\CPb$ containing a square $0$ genus two surface. Since the exotic $\CP\#4\CPb$ constructed above contains such a genus two surface, by normally connect-summing our exotic $\CP\#4\CPb$
with Xiao's genus two fibration on $\mathbb{T}^2 \times \mathbb{S}^2 \#3\CPb$ along it's regular fiber $\bar{\Sigma}_2$ yields to an exotic $3\CP\#6\CPb$. Notice that the fundamental group of Xiao's fibration is carried by its regular fiber and the fibration admits a sphere section. Hence, any meridian of a regular fiber is nullhomotopic in the complement of a regular fiber, and thus a triviality of the fundamental group is straightforward to check by Seifert-Van Kampen's theorem. The remaining details can be filled in as in the previous proofs.  
\end{proof}

 \section*{Acknowledgments} The authors would like to thank Yoshihisa Sato and Susumu Hirose for comments on our work. A. Akhmedov was partially supported by NSF grants, DMS-1005741, Sloan Research Fellowship, and  Guillermo Borja award from the University of Minnesota. N. Monden was partially supported by Grant-in-Aid for Young Scientists (B) (No. 13276356), Japan Society for the Promotion of Science.


\begin{thebibliography}{99}

\bibitem{A1}  A. Akhmedov, Construction of symplectic cohomology $\mathbb{S}^{2}\times \mathbb{S}^{2}$, {\it Proc. of G\"{o}kova Geom. Topol. Conf.} {\bf 14} (2007), 36--48.

\bibitem{A} A. Akhmedov, 
\textit{Small exotic\/ $4$-manifolds}, 
Algebr. Geom. Topol. \textbf{8} (2008), 1781--1794.  

\bibitem{ABP}  A. Akhmedov, I. Baykur, and B.~D. Park: Constructing infinitely many smooth structures on small $4$-manifolds, J. Topol., {\bf 2} (2008), 409--428.

\bibitem{AMon} A. Akhmedov and N. Monden,  \textit{Constructing Lefschetz fibrations via daisy substitution}, Kyoto J. Math,  \textbf{56} to appear (2015).

\bibitem{AM} A. Akhmedov and N. Monden, \textit{Lefschetz fibrations via monodromy substitutions}, preprint.


\bibitem{AP1}  A. Akhmedov and B. D. Park,
\textit{Exotic smooth structures on small\/ $4$-manifolds}, 
Invent. Math. \textbf{173} (2008), 209--223.

\bibitem{AP2}  A. Akhmedov and B. D. Park,
\textit{Exotic smooth structures on small\/ $4$-manifolds with odd signatures}, 
Invent. Math. \textbf{181} (2010), 577--603. 

\bibitem{AP} A. Akhmedov and J. Y. Park, \textit{Lantern substitution and new symplectic 4-manifolds with $b_{2}^{+} = 3$}, Math. Res. Lett., \textbf{21} (2014), no. 1, 1--17.

\bibitem{AO} A. Akhmedov and B. Ozbagci, \textit{Exotic Stein fillings with arbitrary fundamental group}, preprint, arXiv:1212.1743. 

\bibitem{AZ} A. Akhmedov and W. Zhang, \textit{The fundamental group of symplectic $4$-manifolds with $b_2^+ = 1$}, preprint, arXiv:1506.08367.




\bibitem{BHPV} W. P. Barth, K. Hulek, C. A. M. Peters, A. Van de Ven.: {\em Compact complex surfaces}, Springer-Verlag, Berlin Heidelberg, Second Enlarged Edition, 2004. 


\bibitem{B} R. I. Baykur, \textit{Minimality
 and fiber sum decompositions of Lefschetz fibrations}, Proc. Amer. Math. Soc. Proc. 12835 (2015). 

\bibitem{BK} R. I. Baykur, and M. Korkmaz, \textit{Small Lefschetz fibrations and exotic $4$-manifolds}, preprint, arXiv:1506.08367.



\bibitem{D1} S. Donaldson, \textit{Lefschetz pencils on symplectic manifolds}, J. Differential Geom. \textbf{53}, 205--236, 1999.

\bibitem{D2} S. Donaldson,  \textit{Lefschetz fibrations in symplectic geometry}, Documenta Math. Extra Volume ICM 1998, \textbf{II}, 309--314.


\bibitem{Dor} J. Dorfmeister, \textit{Kodaira Dimension of Fiber Sums along Sphere}, Geom. Dedicata, to appear, arXiv:1008.4447. 

\bibitem{DZ} J. G. Dorfmeister, W. Zhang, \textit{The Kodaira dimension of Lefschetz fibrations},
Asian J. Math. \textbf{13}, 2009, 341--358.

\bibitem{E} H. Endo, \textit{Meyer's signature cocyle and hyperelliptic fibrations}, 
Math. Ann., \textbf{316} (2000), 237--257.

\bibitem{EG} H. Endo and Y. Gurtas, \textit{Lantern relations and rational blowdowns}, 
Proc. Amer. Math. Soc. \textbf{138} (2010), no. 3, 1131--1142.

\bibitem{EMVHM} H. Endo, T. E. Mark, and J. Van Horn-Morris, \textit{Monodromy substitutions and rational blowdowns}, 
J. Topol. \textbf{4} (2011), 227--253.


\bibitem{FM} B. Farb and D. Margalit, \textit{A premier on mapping class groups}, Princeton Math.  Series \textbf{49} Princeton University Press (2011).

\bibitem{freedman}  M. H. Freedman,  \textit{The topology of four-dimensional manifolds}, 

\bibitem{FS1} R. Fintushel and R. Stern, \textit{Rational blowdowns of smooth $4$-manifolds}, J. Differential Geom. \textbf{46} (1997), 181--235.

\bibitem{FS2} R. Fintushel and R. Stern, \textit{Families of simply connected 4-manifolds with the sam}, Topology \textbf{43} (2004), 1449--1467.


\bibitem{gompf}  R. E. Gompf, 
\textit{A new construction of symplectic manifolds},
Ann. of Math. \textbf{142} (1995), 527--595.


\bibitem{GS}  R. E. Gompf and A. I. Stipsicz, 
\textit{$4$-Manifolds and Kirby Calculus}, Graduate Studies in
Mathematics, vol. 20, Amer. Math. Soc., Providence, RI, 1999.



\bibitem{I} N.V. Ivanov, \textit{Mapping class groups}, Handbook of Geometric Topology, North-Holland, Amsterdam, (2002), 523--633.

\bibitem{Jo} D. Johnson, \textit{Homeomorphisms of a surface which act trivially on homology}, Proc. Amer. Math. Soc. \textbf{75} (1979), 119--125.



\bibitem{LeBrun} C. LeBrun, \textit{Kodaira dimension and the Yamabe problem}, 
Comm. Anal. Geom. \textbf{7} (1999), 133--156.

\bibitem{Li2006} T.J. Li, \textit{ The Kodaira dimension of symplectic\/ $4$-manifolds}, Proc. Clay Inst. \textbf{5} (2006), 249--261.

\bibitem{li}  T.J. Li,
\textit{Symplectic\/ $4$-manifolds with Kodaira dimension zero},
J. Differential Geom. \textbf{74} (2006), 321--352.

\bibitem{LY} T.J. Li and S-T. Yau, \textit{Embedded surfaces and Kodaira dimension}, ICCM (2007).

\bibitem{Mar} D. Margalit, \textit{A lantern lemma}, Alg. Geom. Top. \textbf{2} (2002), 1179--1195.

\bibitem{Ma1} Y. Matsumoto, \textit{On 4-manifolds fibered by tori II}, 
Proc. Japan Acad., \textbf{59A} (1983), 100-103.

\bibitem{Ma} Y. Matsumoto, \textit{Lefschetz fibrations of genus two --- a topological approach}, 
Topology and Teichm$\ddot{\textnormal{u}}$ller spaces (Katinkulta, 1995), 123--148, World Sci. Publ., River Edge, NJ, 1996.

\bibitem{McDuffS} D. McDuff, \textit{The structure of rational and ruled symplectic\/ $4$-manifolds}, J. Amer. Math. Soc. \textbf{3} (1990), 679--712.

\bibitem{OS} B. Ozbagci and A.I. Stipsicz, \textit{Noncomplex smooth $4$-manifolds
with genus-$2$ Lefschetz fibrations}, Proc. Amer. Math. Soc., 128
(2000), no. 10, 3125--3128.

\bibitem{OS} B. Ozbagci and A. Stipsicz; \emph{Contact 3-manifolds with infinitely many Stein fillings,} 
Proc. Amer. Math. Soc. \textbf{132} (2004), 1549--1558.




\bibitem{Sato} Y. Sato, \textit{The necessary condition on the fiber-sum decomposability of genus-2 Lefschetz fibrations}, Osaka J. Math. \textbf{47} (2010), no.4, 949--963.

\bibitem{taubes}  C. H. Taubes, \textit{The Seiberg-Witten invariants and symplectic forms}, Math. Res. Lett. \textbf{1} (1994), 809--822.


\bibitem{U} M. Usher, Minimality and Symplectic Sums, {\it Int. Math. Res. Not.} (2006), Art ID 49857, 17 pp.

\bibitem{Xiao} G. Xiao, \textit{Surfaces fibr\'ee en courbes de genre deux}, Lecture Notes in Mathematics, \textbf{1137} Springer-Verlag, Berlin, 1985.


\end{thebibliography}
\end{document}